\documentclass[11pt,a4paper]{article}

\usepackage[T1]{fontenc}
\usepackage{lmodern}
\usepackage{amsmath,amssymb,amsthm,mathtools,bm}
\usepackage{xcolor}
\usepackage[margin=1in]{geometry}
\usepackage{microtype}
\usepackage[colorlinks=true,linkcolor=blue,citecolor=green!60!black,urlcolor=blue]{hyperref}

\newcommand{\vct}[1]{\boldsymbol{#1}}
\newcommand{\R}{\mathbb{R}}
\newcommand{\Z}{\mathbb{Z}}

\newcommand{\conv}{\operatorname{conv}}
\newcommand{\aff}{\operatorname{aff}}
\newcommand{\supp}{\operatorname{supp}}

\newcommand{\Vol}{\operatorname{Vol}}
\newcommand{\NVol}{\operatorname{NVol}}
\newcommand{\SVol}{\operatorname{SVol}}
\newcommand{\Ehr}{\operatorname{Ehr}}

\theoremstyle{plain}
\newtheorem{theorem}{Theorem}[section]
\newtheorem{lemma}[theorem]{Lemma}
\newtheorem{proposition}[theorem]{Proposition}
\newtheorem{corollary}[theorem]{Corollary}
\newtheorem{conjecture}[theorem]{Conjecture}

\theoremstyle{definition}
\newtheorem{definition}[theorem]{Definition}
\newtheorem{example}[theorem]{Example}
\newtheorem{remark}[theorem]{Remark}
\newtheorem{problem}[theorem]{Problem}

\usepackage[silent, english]{babel}
\usepackage[autostyle, english = american]{csquotes}
\MakeOuterQuote{"}

\title{Counting Lattice Points in Minkowski Sums of Cross Polytopes}

\author{
    Ziyi Dai\textsuperscript{\hyperlink{author1}{1}}\and
    Qilin Hou\textsuperscript{\hyperlink{author2}{2}}\and
    Zhiyuan Liu\textsuperscript{\hyperlink{author3}{3}}\and
    Warut Thawinrak\textsuperscript{\hyperlink{author4}{4}}\and
    Hongyu Wang\textsuperscript{\hyperlink{author5}{5}}
}

\date{}

\begin{document}
\maketitle

\begin{abstract}
Motivated by Postnikov’s study of lattice-point enumeration in Minkowski sums of simplices, we investigate lattice points in Minkowski sums of cross polytopes and establish analogous results, together with several related consequences. In particular, we introduce the support-enumerator associated with Postnikov’s notion of draconian sequences and show that it coincides with the $h^*$-polynomial of the corresponding root polytope. This provides a new interpretation of the $h^*$-polynomial and yields a simple method for computing the volume of the corresponding polytope. By exploiting the symmetry of these root polytopes, we further establish a duality property for support-enumerators, which in turn provides a proof of a conjecture by Athanasiadis and Chapoton concerning the $h$-polynomials of preorders. Consequently, we obtain a formula for the number of lattice points in Minkowski sums of cross polytopes in terms of draconian sequences and show that these polytopes are Ehrhart positive. Furthermore, this formula leads to analogous expressions for the number of lattice points on their boundaries and for their surface volumes. 
\end{abstract}


\vspace{0.3in}

\section{Introduction}\label{sec:intro}
Many important families of polytopes can be constructed as Minkowski sums (and differences) of polytopes. Consequently, studying the lattice-point enumeration of such Minkowski sums can reveal fruitful combinatorial properties of these families and open up new directions for further research. \textit{Type A generalized permutohedra} \cite{Postnikov2009}, also known simply as \textit{generalized permutohedra}, provide a prominent example. These polytopes can be constructed as Minkowski sums and differences of simplices \cite{Postnikov2009,Ardila2010} and have deep connections to several areas of mathematics, including the theory of matroids and Weyl groups. Postnikov’s influential work on their lattice-point enumeration \cite{Postnikov2009} reveals several intriguing properties of this family, which in turn inspired a wealth of subsequent research. Among these developments is the study of their type B counterparts in \cite{Euretall2024,thawinrak2025AtoB}. In fact, the study of Minkowski sums of cross polytopes in this paper is also motivated by Postnikov’s work. We elaborate on this motivation, introduce the key definitions, and highlight the main contributions of our work below.

Let $[n]$ denote the set $\{1,2,\ldots,n\}$. For $i\in[n]$, we define
$\vct{e}_i:=(0,\ldots,0,1,0,\ldots,0)\in\R^n$ to be the point with $1$ at the $i^{\mathrm{th}}$ component and $0$ at all other components. For $I\subset[n]$, we let
\[
\Delta_I^0:=\conv(\vct{0},\vct{e}_i\mid i\in I)\subset\R^n
\]
be the corresponding \textit{simplex}.

Let
\[
G=(L\sqcup R,E),\qquad
L=\{\ell_1,\ldots,\ell_m\},\qquad
R=\{r_1,\ldots,r_n\}
\]
be a bipartite graph on $m$ left vertices $\ell_1, \dots, \ell_m$ and $n$ right vertices $r_1, \dots, r_n$. We denote by $G^*$ the bipartite graph obtained by swapping the left and right vertex sets of $G$. For $j\in[m]$, we define
\begin{equation}
I_j:=\{i\in[n]\mid \ell_j\text{ is adjacent to }r_i\}
\label{eq:Ij}
\end{equation}
to be the set of (the indices of) the neighbors of $\ell_j.$ For $(y_1,\ldots,y_m)\in\R^m$, we define the polytope
\begin{equation}
P_G(y_1,\ldots,y_m):=y_1\Delta_{I_1}^0+\cdots+y_m\Delta_{I_m}^0.
\label{eq:PG}
\end{equation}
These polytopes form the family of \emph{polymatroids} \cite{edmonds2003submodular,fujishige2005submodular} and a subfamily of \emph{type B generalized permutohedra} \cite{Bastidas2021,Euretall2024}. Moreover, they can be realized simply as type A generalized permutohedra \cite{thawinrak2025AtoB}. Thus, the theory of type A generalized permutohedra developed by Postnikov in \cite{Postnikov2009} provides sufficient tools for enumerating their lattice points. 

By viewing $P_{G^*}(1,\dots,1)$ as the dual polytope of $P_{G}(1,\dots, 1),$ Postnikov shows that this duality preserves the number of lattice points. Consequently, this result allows Postnikov to derive a formula for the number of lattice points in $P_G(y_1,\dots, y_m)$ in terms of $G$-\emph{draconian sequences}, which we define in the following definition. An important consequence of Postnikov’s formula is that it makes \textit{Ehrhart positivity} of $P_G(y_1,\dots,y_m)$ immediately apparent for all nonnegative integers $y_1,\dots,y_m$.

\begin{definition}\label{def: draconion-sequence}
A sequence $(a_1,\ldots,a_m)$ of nonnegative integers is a \emph{$G$-draconian sequence} if for every nonempty subset $J\subset[m]$,
\[
\sum_{i\in J}a_i\le \left|\bigcup_{i\in J}I_i\right|.
\]
The set of all $G$-draconian sequences is denoted by $D(G)$.
\end{definition}

The purpose of this paper is to develop an analogous theory for
Minkowski sums of \textit{cross polytopes} and establish related consequences. We introduce further key notations and highlight our main contributions as follows.

For a nonempty
$I\subset[n]$, let
\[
\diamondsuit_I=\conv(\pm\vct{e}_i\mid i\in I) \subset \R^n
\]
denote the corresponding \textit{cross polytope}. Given a bipartite graph
$G$, we consider
\begin{equation}\label{eq:cross-polytope}
   Q_G(k_1,\ldots,k_m)
=
k_1\diamondsuit_{I_1}+\cdots+k_m\diamondsuit_{I_m}. 
\end{equation}
These polytopes form a natural type $B$ analogue of the Minkowski sums
of simplices. Indeed, the intersection of
$Q_G(k_1,\ldots,k_m)$ with the nonnegative orthant is precisely $P_G(k_1,\dots, k_m).$ This close relationship between simplices and cross polytopes makes the lattice-point counting problem
both natural and nontrivial. Our first step toward this enumeration is the introduction of the
\emph{support-enumerator} of $D(G)$ 
\begin{equation}
F_G(z):=\sum_{\vct{a}\in D(G)}z^{s(\vct{a})}
\label{eq:FG}
\end{equation}
where $s(a)$ is
the number of positive coordinates of $a$. As one of our main results, we show that the support-enumerator is not only an auxiliary counting device, but also coincides with the $h^*$-polynomial of a certain root polytope $R_{\widehat G}$ associated with the \textit{augmented bipartite graph} $\widehat G$ of $G$. See Section \ref{sec:h-polynomial} for the precise definitions.
\begin{theorem}[Realization as $h^*$-polynomial]\label{thm:F_G-h-polynomial}
For every bipartite graph $G$ on $m$ left and $n$ right vertices, the Ehrhart series of the root polytope $R_{\widehat G}$ defined in \eqref{eq:root-polytope} is given by
\begin{equation}
\Ehr_{\mathcal R_{\widehat G}}(z)
=\frac{F_G(z)}{(1-z)^{m+n+1}},
\label{eq:root-ehrhart}
\end{equation}
where $F_{G}(z)$ is the support-enumerator defined in \eqref{eq:FG}.
Equivalently, the $h^*$-polynomial of $R_{\widehat G}$ is
\begin{equation}
h^*_{\mathcal R_{\widehat G}}(z)=F_G(z).
\label{eq:hstar-F}
\end{equation}
\end{theorem}
This result gives a new combinatorial interpretation of the
$h^*$-polynomial of a class of root polytopes, complementing and simplifying its existing interpretations as an \emph{interior polynomial} established by K\'alm\'an and
Postnikov \cite{KalmanPostnikov2017}, and certain statistics on the \textit{dissecting spanning trees} of $G$ introduced by K\'alm\'an and T\'othm\'er\'esz in \cite{KalmanTothmeresz2023}. Moreover, it connects the statistics on the supports with several other combinatorial objects, raising the question of what further combinatorial interpretations may be encoded by these support-enumerators.

By employing the symmetry of the root polytope $R_{\widehat G},$ we deduce the following duality for support-enumerators.

\begin{theorem}[Duality for support-enumerator]\label{thm:duality-support-enumerator}
For every bipartite graph $G$, we have
\begin{equation}
F_G(z)=F_{G^*}(z),
\label{eq:duality}
\end{equation}
where $F_{G}(z)$ is the support-enumerator defined in \eqref{eq:FG}, and $G^*$ is the bipartite graph obtained from $G$ by swapping its left and right vertex sets. 
\end{theorem}

In \cite{chapoton2026}, Athanasiadis and Chapoton study a special class of Minkowski sums of simplices that can be described in terms of \emph{preorders}. They define the $h$-polynomial of a preorder, which is equivalent to our notion of a support-enumerator, and propose several conjectures concerning its properties. We found that Theorem \ref{thm:duality-support-enumerator} 
gives a proof of their Conjecture 5.4 in \cite{chapoton2026} regarding the duality of the $h$-polynomials of preorders.

By viewing $Q_{G^*}(1,\dots, 1)$ as the dual polytope of $Q_G(1,\dots, 1)$, Theorem \ref{thm:duality-support-enumerator} implies that this duality preserves the number of lattice points, thereby providing an analogue of the corresponding result for Minkowski sums of simplices. This, in turn, leads to the following formula for the number of lattice points in Minkowski sums of cross polytopes.

\begin{theorem}[Lattice-point counting formula]\label{thm:general-formula-for-lattice-points}
Let $G$ be a bipartite graph on $m$ left vertices and $n$ right vertices. Then, for every $k_1,\ldots,k_m\in\Z_{\ge0}$,
\begin{equation}
|Q_G(k_1,\ldots,k_m)\cap\Z^n|
=\sum_{\vct{a}\in D(G)}\prod_{i=1}^m c(a_i,k_i),
\label{eq:general-count}
\end{equation}
where $D(G)$ is the set of $G$-draconian sequences in Definition \ref{def: draconion-sequence} and $c(a,k)$ is an integer satisfying, for all $k \in \Z_{\geq 0}$,
\begin{equation*}
c(0,k) = 1 \quad \text{and} \quad c(a,k)=\sum_{r=1}^{\min(a,k)}2^r\binom{k}{r}\binom{a-1}{r-1} \text{ for } a\geq 1.
\end{equation*}
\end{theorem}

We note that the function $c(a,k)$ in Theorem \ref{thm:general-formula-for-lattice-points} is a polynomial in $k$ for a given nonnegative integer $a$. By establishing results regarding its coefficients, we are able to easily deduce the Ehrhart positivity for Minkowski sums of polytopes. 

\begin{theorem}[Ehrhart positivity]\label{thm:positivity}
Let $G$ be a bipartite graph on $m$ left vertices. Then, the polytope $Q_G(k_1, \dots, k_m)$ is Ehrhart positive for all nonnegative integers $k_1, \dots, k_m$, i.e., every coefficient of its Ehrhart polynomial is strictly positive.
\end{theorem}

We note that since
$Q_G(k_1,\ldots,k_m)$ belongs to the family of type $B$ generalized
permutohedra, existing formulas for lattice-point enumeration in \cite{Euretall2024} and \cite{thawinrak2025AtoB} can be applied to compute its number of lattice points. However, these formulas do not make the Ehrhart positivity particularly
transparent. Thus, our formula is better suited for this family of polytopes in this sense. See Remark \ref{rem:other-formulas} for a more detailed discussion on this positivity.

The properties of $c(a,k)$ allow us to go beyond deducing the Ehrhart positivity. By showing that the parity of $c(a,k)$ is the same as the parity of $a$, we further derive formulas in similar forms to \eqref{eq:general-count} for the number of lattice points on the boundary and for the surface volume.

\subsection*{Paper Organization}

In Section \ref{sec:prelim}, we introduce basic notations regarding polytopes and review preliminary results from Ehrhart theory, along with existing techniques for counting lattice points in Minkowski sums of simplices. In Section \ref{sec: sums-of-cross-polytopes}, we establish fundamental properties of Minkowski sums of cross polytopes that can be deduced directly from the study of Minkowski sums of simplices. In Section \ref{sec:lttice-point-enumeration}, we investigate the lattice-point enumeration of Minkowksi sums of cross polytopes, and derive related results. Finally, in Section \ref{sec:further-questions}, we present several potential research directions concerning other aspects of Minkowski sums of cross polytopes that remain unexplored.    

\section{Preliminaries}\label{sec:prelim}
\subsection{Polytopes}

A \emph{polytope} $P$ in $\R^n$ can be defined as the convex hull of a finite set of points in $\R^n$, i.e.,
\[
P=\conv(\vct{x}_1,\ldots,\vct{x}_k)
 :=\left\{\lambda_1\vct{x}_1+\cdots+\lambda_k\vct{x}_k
 \mathrel{\Big|}
 \lambda_1+\cdots+\lambda_k=1,\ \lambda_i\ge 0\text{ for all }i\in[k]\right\}.
\]
By the Minkowski-Weyl Theorem \cite{Minkowski1968, Weyl1934}, one can equivalently define a polytope using a system of inequalities. The dimension of $P$, denoted by $\dim(P)$, is the dimension of the affine span $\aff(P)$ of $P$. We denote by $\Vol(P)$ the volume of $P$ with respect to the lattice in $\aff(P)$. The normalized volume $\NVol(P)$ of $P$ is then defined by $\NVol(P) := \dim(P)!\Vol(P).$

A subset $F$ of a polytope $P$ is said to be a \emph{face} of $P$ if there exists a hyperplane $H$ such that $P$ lies on one side of $H$ and $F=P\cap H$. A face of dimension $0$ is called a \emph{vertex}, a face of dimension $1$ is called an \emph{edge}, and a face of dimension $\dim(P)-1$ is called a \emph{facet}. The \emph{surface volume}, denoted by $\SVol(P)$, of $P$ is the sum of the volume $\Vol(F)$ over all facets $F$ of $P$. We say that $P$ is a \emph{lattice} (or \emph{integral}) \emph{polytope} if all of its vertices have integral coordinates.

Given two nonempty polytopes $P_1,P_2\subset\mathbb{R}^n$, their \emph{Minkowski sum}, denoted by $P_1+P_2$, is defined as
\[
P_1+P_2
:=
\{\vct{x}_1+\vct{x}_2\in\mathbb{R}^n\mid \vct{x}_1\in P_1,\ \vct{x}_2\in P_2\}.
\]
If $U$ and $V$ denote the sets of vertices of $P_1$ and $P_2$, respectively, then
\[
P_1+P_2
:=
\conv\bigl(\vct{u}+\vct{v}\mid \vct{u}\in U,\ \vct{v}\in V\bigr).
\]
In particular, the Minkowski sum of two polytopes is again a polytope.

The \emph{Minkowski difference} of $P_2$ in $P_1$, denoted by $P_1-P_2$, is defined as
\[
P_1-P_2
=
\{\vct{x}\in\mathbb{R}^n\mid \vct{x}+P_2\subset P_1\}.
\]
Equivalently, $P_1-P_2$ is the set of translation vectors that place $P_2$ inside $P_1$, and is again a polytope. It is important to note that Minkowski difference is neither commutative nor associative in general. For instance, we have $(P_1 + P_2) - P_2 = P_1$, but $(P_1 - P_2) + P_2$ may not equal $P_1$, or may even be undefined if $P_1 - P_2$ is empty. Thus, when working with both Minkowski sums and differences, it is important to specify the order of the operations. Given nonempty polytopes $P_1,\ldots,P_m\subset\mathbb{R}^n$ and signs $\delta_i\in\{1,-1\}$ for $i\in[m],$ we define
\[\sum^{m}_{i = 1}\delta_i P_i := Q_1 - Q_2 \text{ where } Q_1 := \sum_{\delta_i = 1} P_i \text{ and } Q_2 := \sum_{\delta_i = -1} P_i.\]
In this paper, we primarily focus on Minkowski sums of polytopes, and only refer to Minkowski differences on a few occasions.
\subsection{Ehrhart Theory}\label{sec:Ehrhart-theory}
Given a nonnegative real number $t$, the $t^{\mathrm{th}}$-dilation of $P$, denoted by $tP$, is the set $\{t\vct{x}\mid \vct{x}\in P\}$. When $t$ is a positive integer, $tP$ equals the Minkowski sum of $t$ copies of $P$. We define 
\[
i(P,t):=|tP\cap\Z^n|
\]
to be the number of lattice points in $tP$ for all $t \in \Z_{\geq 0}$. A fundamental result of Ehrhart theory \cite{Ehrhart1962} establishes that, for every lattice polytope $P$, the function $i(P,t)$ is a polynomial in $t\in\Z_{\geq 0}$ of degree equal to $\dim(P)$, the dimension of $P$. We call $i(P,t)$ the \emph{Ehrhart polynomial} of $P$. It encodes several geometric and combinatorial properties of $P$. The following are well-known properties of Ehrhart polynomials.
\begin{enumerate}
    \item\label{itm: constant-term} The constant term of $i(P,t)$ is 1.
    \item\label{itm: leading-coefficient} The leading coefficient of $i(P,t)$ equals $\Vol(P)$, the volume of $P$.
    \item\label{itm: second-leading-coefficient} The second-leading coefficient (the coefficient of $t^{\dim(P)-1}$) of $i(P,t)$ equals $\SVol(P)/2$.
    \item\label{itm: recipocity} The number of lattice points in the (relative) interior of $tP$ equals $(-1)^{\dim(P)}i(P,-t).$
\end{enumerate}

We note that property \ref{itm: recipocity} is commonly referred to as the \emph{reciprocity} of Ehrhart polynomials. Properties \ref{itm: constant-term}--\ref{itm: second-leading-coefficient} might lead us to believe that all the coefficients of Ehrhart polynomials are nonnegative. However, this is not the case for some polytopes of dimension greater than 2. We refer the reader to the survey \cite{Liu2019} by Liu regarding the positivity of the coefficients of Ehrhart polynomials. When every coefficient of $i(P,t)$ is positive, we say that $P$ is \emph{Ehrhart positive}. As discussed in the survey by Liu, Ehrhart positivity has been established for several important classes of lattice polytopes. Determining whether a given family is Ehrhart positive is an intriguing problem in Ehrhart theory.

For every lattice polytope $P$ of dimension $d$, one can always express
\begin{align}\label{eq:h-polynomial-Ehrhart-Polynomial}
    i(P,t) = h^*_0\binom{t+d}{d} + h^*_1\binom{t+d-1}{d} + \cdots + h^*_d\binom{t}{d}
\end{align}
for some real numbers $h^*_0, h^*_1, \dots, h^*_d.$ This is due to the fact that the binomials $\binom{t+d}{d}, \binom{t+d-1}{d}, \dots, $ $\binom{t}{d}$ form a basis of the space 
\[\R[t]_{\leq d} :=\{a_0 + a_1t+\cdots + a_dt^d\mid a_0 \in \R \text{ and } a_i \in \R \text{ for all } i \in [d]\}\]
of polynomials of degree at most $d$ over $\R$ (in variable $t$). Since $i(P,t)$ is an integer for all $t \in \Z_{\geq 0}$, we further have that the coefficients $h^*_0, h^*_1, \dots, h^*_d$ are all integers. In fact, a fundamental theorem by Stanley \cite{Stanley1980} asserts that these coefficients are nonnegative integers. We call the polynomial
\[h^*_P(z):= h^*_0 + h^*_1z + \cdots + h_d^*z^d\]
the \emph{$h^*$-polynomial} of $P$. By definition, $h^*_P(z)$ is a polynomial of degree at most $d$ with nonnegative integral coefficients. Since the leading coefficient of $i(P,t)$ equals the volume of $P$, one sees from \eqref{eq:h-polynomial-Ehrhart-Polynomial} that 
\[h^*_P(1) = h^*_0 + h^*_1 + \cdots + h^*_d = d!\Vol(P) = \NVol(P).\]
That is, for every nonempty lattice polytope $P$, the sum of the coefficients of $h^*_P(z)$ is equal to the normalized volume of $P$ and is a positive integer.

The \emph{Ehrhart series} of a lattice polytope $P$ is defined as the generating function
\[\Ehr_P(z):= \sum_{t \geq 0}|tP\cap \Z^n|z^t = \sum_{t \geq 0}i(P,t)z^t.\]
By \eqref{eq:h-polynomial-Ehrhart-Polynomial}, we equivalently have
\[\Ehr_P(z):= \frac{h^*_P(z)}{(1-z)^{\dim(P)+1}}.\]

Two lattice polytopes $P \subset \R^n$ and $Q \subset \R^m$ are said to be \emph{integrally equivalent} if there exists an invertible affine transformation from $\mathrm{aff}(P)$ to $\mathrm{aff}(Q)$ that preserves the lattice points in $P$ and $Q$. Two lattice polytopes that are integrally equivalent have identical Ehrhart polynomials and, consequently, the same $h^*$-polynomial and volume.

\subsection{Minkowski Sums of Simplices}\label{sec: sums-of-simplices}

Recall the notations and definitions concerning bipartite graphs and the polytopes $P_G(y_1,\dots, y_m)$ introduced in Section \ref{sec:intro}. We now provide examples, discuss the basic properties, and review existing techniques for counting lattice points in these polytopes that motivate this work.

\begin{figure}
    \centering
\includegraphics[width=1\linewidth]{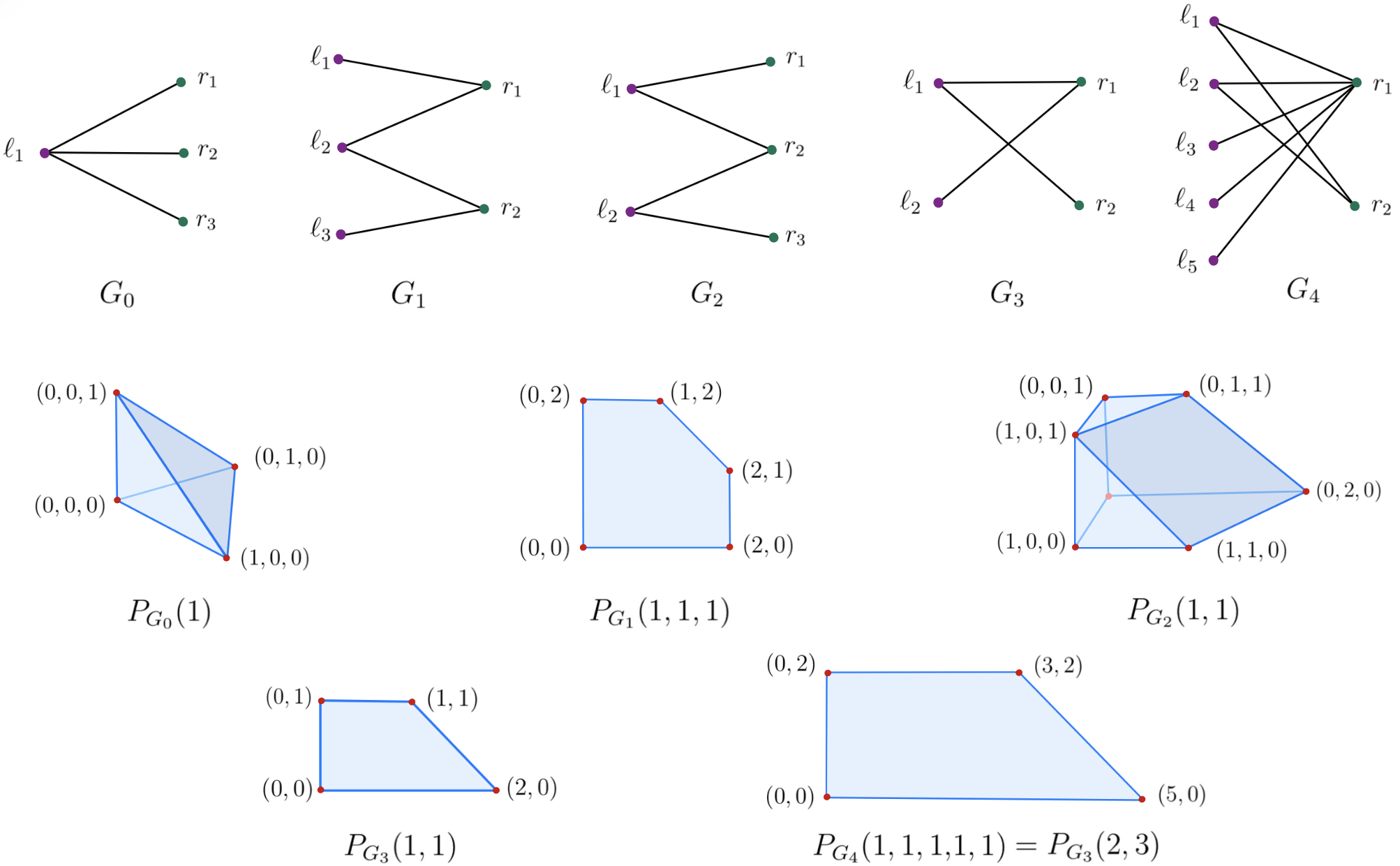}
    \caption{Minkowski sums of simplices and their corresponding bipartite graphs}
    \label{fig:graphs-simplices}
\end{figure}

\begin{example}\label{ex:P_G}
 Let $G_0, G_1, \dots, G_4$ be the bipartite graphs shown in Figure \ref{fig:graphs-simplices}. Then, the following polytopes are also shown in Figure \ref{fig:graphs-simplices}.
 \begin{align*}
     P_{G_0}(1) &= \Delta^0_{[3]} \subset \R^3,\\
     P_{G_1}(1,1,1) &= \Delta^0_{\{1\}} +\Delta^0_{[2]} + \Delta^0_{\{2\}} \subset \R^2,\\
     P_{G_2}(1,1) &= \Delta^0_{[2]} + \Delta^0_{\{2,3\}} \subset \R^3,\\
     P_{G_3}(1,1) &= \Delta^0_{[2]} + \Delta^0_{\{1\}}\subset \R^2,\\
     P_{G_4}(1,1,1,1,1) &= 2\Delta^0_{[2]} + 3\Delta^0_{\{1\}} = P_{G_3}(2,3) \subset \R^2.
 \end{align*}
 \end{example}
 
The polytope $P_G(y_1,\ldots,y_m)$ is known to be a projection of a type A generalized permutohedron and can be viewed simply as such a permutohedron (See \cite[Lemma 3.7]{thawinrak2025AtoB}). This allows us to apply tools and techniques from the study of type A generalized permutohedra by Postnikov in \cite{Postnikov2009} to $P_G(y_1, \dots, y_m)$. 

\begin{lemma}[\cite{Postnikov2009}]\label{lem:inequalities-P_G}
The polytope $P_G(y_1,\ldots,y_m)$ is given by the set of all $\vct{x}\in\R^n$ such that $0\le x_i$ for all $i\in[n]$, and for all non-empty subsets $J\subset[n]$,
\[
\sum_{j\in J}x_j\le \sum_{I_r\cap J\ne\varnothing}y_r.
\]
\end{lemma}
The draconian sequences introduced in Definition \ref{def: draconion-sequence} are combinatorial objects that play a crucial role in counting lattice points throughout this paper. The following four lemmas demonstrate their significance and show how they arise in this counting process.
\begin{remark}
    The notion of draconian sequence originally defined by Postnikov in \cite{Postnikov2009} slightly differs from our notion in Definition \ref{def: draconion-sequence}. In the original notion, a $G$-draconian sequence would have been defined as a sequence of nonnegative integers $(a_0, a_1, \dots, a_m)$ such that $a_0 + a_1 + \cdots + a_m = n$ and $a_1, \dots, a_m$ are as described in Definition \ref{def: draconion-sequence}. This slight abuse of notation will cause no ambiguity, as the original notion does not appear elsewhere in the paper.
\end{remark}
     
For a given bipartite graph $G$, we denote by $G^*$ the dual bipartite graph of $G$ obtained by swapping the left and right sets of vertices.

\begin{lemma}[\cite{Postnikov2009}]\label{lem: draconion-sequences-dual}
The $G$-draconian sequences are the lattice points in $P_{G^*}(1,\ldots,1),$ i.e., 
\[D(G) = P_{G^*}(1,\ldots,1)\cap \Z^m.\]    
\end{lemma}

\begin{example}\label{ex:draconian-sequences}
    By Figure \ref{fig:graphs-simplices}, we see that
    \begin{align*}
    P_{G^*_0}(1,1,1) &= 3\Delta^0_{[1]} = [0,3] \subset \R,
    &\qquad P_{G_1^*}(1,1) &= P_{G_2}(1,1) \subset \R^3,\\
    P_{G_2^*}(1,1,1) &= P_{G_1}(1,1,1) \subset \R^2,
    &\qquad P_{G_3^*}(1,1) &= P_{G_3}(1,1) \subset \R^2.
\end{align*}
Thus, we have
    \begin{align*}
        D(G_0) &= P_{G^*_0}(1,1,1)\cap \Z = [0,3]\cap \Z= \{0,1,2,3\},\\
        D(G_1) &= P_{G_2}(1,1)\cap \Z^3 \\
        &= \{(0,0,0), (1,0,0), (0,1,0),(0,0,1),(1,1,0),(1,0,1),(0,1,1),(0,2,0)\},\\
        D(G_2) &= P_{G_1}(1,1,1)\cap \Z^2 = \{(0,0), (1,0),(0,1),(2,0),(1,1),(0,2),(2,1),(1,2)\},\\
        D(G_3) &= P_{G_3}(1,1)\cap \Z^2 = \{(0,0), (1,0),(0,1),(2,0),(1,1)\}.
    \end{align*}
\end{example}

One may view $P_{G^*}(1,\ldots,1)$ as the "dual" polytope of $P_G(1,\ldots,1)$ and vice versa. Thus, Lemma \ref{lem: draconion-sequences-dual} states that the lattice points in the dual of $P_{G}(1,\ldots,1)$ are precisely the $G$-draconian sequences, and that $P_G(1,\dots, 1)\cap \Z^n = D(G^*)$.  The next lemma shows that this duality preserves the number of lattice points.

\begin{lemma}[\cite{Postnikov2009}]\label{lem: lattice-points-draconion-sequences}
The number of lattice points in $P_G(1,\ldots,1)$ equals the number of $G$-draconian sequences. That is,
\[
|P_G(1,\ldots,1)\cap\Z^n|=|P_{G^*}(1,\ldots,1)\cap\Z^m|.
\]
\end{lemma}

Readers can easily see the result of Lemma \ref{lem: lattice-points-draconion-sequences} demonstrated using Example \ref{ex:draconian-sequences} and Figure \ref{fig:graphs-simplices}. As a consequence, we obtain the following formula for the number of lattice points.

\begin{lemma}[\cite{Postnikov2009}]\label{lem:lattice-points-formula-P_G}
Let $y_1,\ldots,y_m$ be nonnegative integers. Then,
\begin{equation}
|P_G(y_1,\ldots,y_m)\cap\Z^n|
=\sum_{\vct{a}\in D(G)}
\binom{y_1+a_1-1}{a_1}\cdots
\binom{y_m+a_m-1}{a_m}.
\label{eq:simplex-count}
\end{equation}
\end{lemma}

\begin{remark}\label{rem:validity-of-negative-coefficients-P_G}
    When some of $y_1, \dots, y_m$ are negative integers, the formula \eqref{eq:simplex-count} in Lemma \ref{lem:lattice-points-formula-P_G} may still apply to $P_G(y_1,\dots,y_m).$ More precisely, if $P_G(y_1,\dots,y_m)$ satisfies
    \begin{equation}\label{eq:Minkowski-summand-P_G}
        P_G(y_1,\dots,y_m) + \sum_{i\,:\,y_i < 0}|y_i|\Delta^0_{I_i} = \sum_{i\,:\, y_i\geq0}y_i\Delta^0_{I_i},
    \end{equation}
    then formula \eqref{eq:simplex-count} remains valid. This follows from the property of type $B$ generalized permutohedra established in \cite[Theorem 6.11]{Bastidas2021} together with the fact that Minkowski sums (and differences) of simplices are \emph{deformations} of a \emph{permutohedron} (see \cite[Remark 6.8]{Postnikov2009}).
\end{remark}

The terms of highest degree of the multivariate polynomial (in variables $y_1, \dots, y_m$) in \eqref{eq:simplex-count} give a volume formula as stated in the next lemma.

\begin{lemma}[\cite{Postnikov2009}]\label{lem: volume-P_G}
   Let $y_1,\ldots,y_m$ be nonnegative real numbers. Suppose that $P_G(y_1,\ldots,y_m)$ is $d$-dimensional.  Then,
\begin{equation}
\Vol(P_G(y_1,\ldots,y_m))
=\sum_{\substack{\vct{a}\in D(G)\\|\vct{a}|=d}}
\frac{y_1^{a_1}\cdots y_m^{a_m}}{a_1! \cdots a_m!},
\label{eq: volume-P_G}
\end{equation}
 where $|\vct{a}| = a_1+\cdots+a_m$ is the sum of the components of $\vct{a}$.
\end{lemma}

Note that one can compute the Ehrhart polynomial of $P_G(y_1,\ldots,y_m)$ by replacing $(y_1,\ldots,y_m)$ in Equation~\eqref{eq:simplex-count} by $(y_1t,\ldots,y_mt)$. It is easy to see that if $y_1,\ldots,y_m$ are nonnegative integers, then $P_G(y_1,\ldots,y_m)$ is Ehrhart positive.

\section{Minkowski Sums of Cross Polytopes}\label{sec: sums-of-cross-polytopes}

In this section, we provide examples of the polytopes $Q_G(k_1,\dots,k_m)$ defined in \eqref{eq:cross-polytope} and use their symmetries, along with the results from Minkowski sums of simplices, to establish their basic properties.
\begin{figure}
    \centering
    \includegraphics[width=1\linewidth]{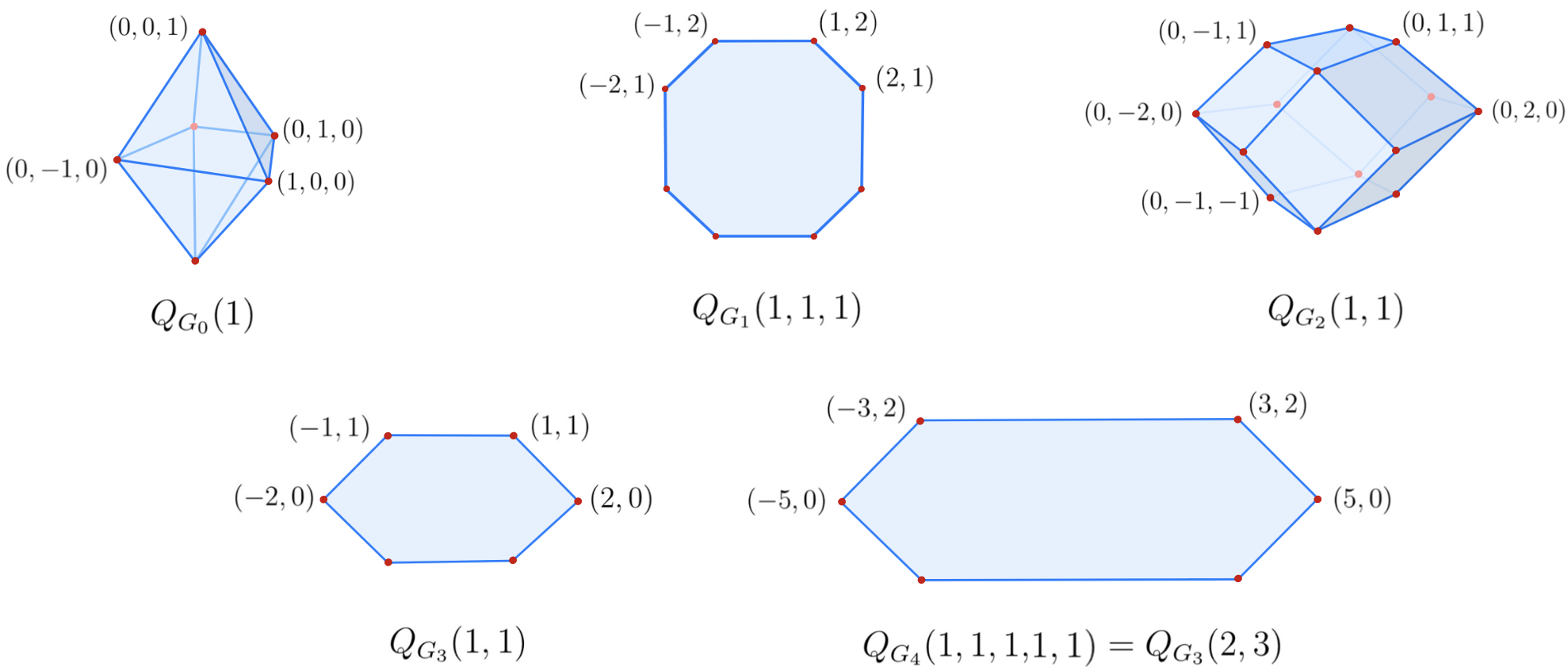}
    \caption{Minkowski sums of cross polytopes}
    \label{fig:cross-polytopes}
\end{figure}

\begin{example}
 Let $G_0, G_1, \dots, G_4$ be the bipartite graphs shown in Figure \ref{fig:graphs-simplices}. Then, the following polytopes are shown in Figure \ref{fig:cross-polytopes}.
 \begin{align*}
     Q_{G_0}(1) &= \diamondsuit_{[3]} \subset \R^3,\\
     Q_{G_1}(1,1,1) &= \diamondsuit_{\{1\}} +\diamondsuit_{[2]} + \diamondsuit_{\{2\}} \subset \R^2,\\
     Q_{G_2}(1,1) &= \diamondsuit_{[2]} + \diamondsuit_{\{2,3\}} \subset \R^3,\\
     Q_{G_3}(1,1) &= \diamondsuit_{[2]} + \diamondsuit_{\{1\}}\subset \R^2,\\
     Q_{G_4}(1,1,1,1,1) &= 2\diamondsuit_{[2]} + 3\diamondsuit_{\{1\}} = Q_{G_3}(2,3) \subset \R^2.
 \end{align*}
 \end{example}

The dimension of $Q_G(k_1,\dots,k_m)$ can be computed in terms of the bipartite graph $G$.
\begin{proposition}\label{prop:dim-of-Q_G}
Let $k_1,\ldots,k_m$ be positive real numbers and write $U := I_1\cup \cdots \cup I_m.$ Then,
\begin{align}\label{eq: dim-of-Q_G}
\dim(Q_G(k_1,\ldots,k_m)) = |U| = \#(\text{right vertex of $G$ whose degree is at least one}).
\end{align}
\end{proposition}
\begin{proof}
By the definition of $Q_G(k_1, \dots, k_m)$, every $(x_1, \dots, x_n) \in Q_G(k_1, \dots, k_m)$ satisfies $x_i = 0$ for all $i \in [n]\setminus U.$ Thus, the dimension of $Q_G(k_1, \dots, k_m)$ is at most $|U|$. Additionally, for every $i \in U$, there exists $j \in [m]$ such that $i \in I_j$. Thus, the points $\vct{0}$ and $k_j\vct{e}_i$ lie in $Q_G(k_1, \dots, k_m)$ for all $i \in U$, and some $j \in [m]$. Thus, $Q_G(k_1, \dots, k_m)$ contains $|U|+1$ affinely independent points. Hence, the dimension of $Q_G(k_1, \dots, k_m)$ is at least $|U|.$ This gives the first equality in \eqref{eq: dim-of-Q_G}. The second equality in \eqref{eq: dim-of-Q_G} follows from the fact that the elements of $U$ are precisely those right vertices of $G$ whose degrees are at least one.
\end{proof}

The next proposition shows that part of $Q_G(k_1, \dots, k_m)$ that lies on the first orthant is precisely the polytope $P_G(k_1, \dots, k_m)$ defined in Section~\ref{sec: sums-of-simplices}. This is also evident in the examples given in Figure \ref{fig:graphs-simplices} and Figure \ref{fig:cross-polytopes}. 

\begin{proposition}\label{prop: Q_G-first-orthant}
Suppose that $k_1,\ldots,k_m$ are nonnegative real numbers. Then,
\[
Q_G(k_1,\ldots,k_m)\cap\R_{\ge 0}^n
=k_1\Delta_{I_1}^0+\cdots+k_m\Delta_{I_m}^0 = P_G(k_1, \dots, k_m).
\]
\end{proposition}

\begin{proof}
As $k_i\Delta^0_{I_i} \subset k_i\diamondsuit_{I_i}$ for all $i \in [m],$ it follows that $Q_G(k_1,\ldots,k_m)\cap\R_{\ge 0}^n
\supseteq P_G(k_1, \dots, k_m).
$

Now consider $\vct{x} = (x_1, \dots, x_n) \in Q_G(k_1,\ldots,k_m)\cap\R_{\ge 0}^n.$ Then $\vct{x} = \vct{x}^{(1)}+\cdots +\vct{x}^{(m)}$ where $\vct{x}^{(j)}:= (x^{(j)}_1,\dots, x^{(j)}_n) \in k_j\diamondsuit_{I_j}$ for all $j \in [m].$ We want to show that $\vct{x} \in P_G(k_1, \dots, k_m).$ By definition, $x_i = x^{(1)}_i + \cdots + x^{(m)}_i \geq 0$ for all $i \in [n].$ For each $i \in [n],$ let $L_i := \{j \in [m] \mid x^{(j)}_i < 0\}$ be the set of the indices of negative summands. We define $a_i := \sum_{j \in [m]\setminus L_i}x^{(j)}_i$ and $b_i := \sum_{j \in  L_i}x^{(j)}_i$ (set $b_i = 0$ if $L_i = \emptyset$). Let
\begin{align*}
y^{(j)}_i :=
    \begin{cases}
    x^{(j)}_i(a_i+b_i)/a_i &\text{ if } j \in [m]\setminus L_i \text{ and } a_i > 0\\
        0&\text{ otherwise }
    \end{cases}.
\end{align*}
Then, $0 \leq y^{(j)}_i \leq |x^{(j)}_i|$ for all $j \in [m]$ and $i \in [n]$. This implies that $\vct{y}^{(j)}:= (y^{(j)}_1, \dots, y^{(j)}_n) \in k_j\Delta^0_{I_j}$ for all $j \in [m]$. Moreover, for every $i \in [n]$, we have 
\[\sum^m_{j = 1}y^{(j)}_i = \sum_{j\in [m]\setminus L_i}y^{(j)}_i = \sum_{j\in [m]\setminus L_i}\frac{x^{(j)}_i(a_i+b_i)}{a_i} = a_i + b_i = x_i.\]
Thus, we have $\vct{x} = \vct{y}^{(1)} + \cdots + \vct{y}^{(m)}.$ Hence, $\vct{x} \in P_G(k_1, \dots, k_m).$ This completes the proof.
\end{proof}

\begin{remark}\label{rem: Q_G-P_G}
A point $\vct{x}=(x_1,\ldots,x_n)$ lies in $\Delta^0_I$ if and only if every point of the form $(\pm x_1,\ldots,\pm x_n)$ lies in $\diamondsuit_I$. Thus, $\vct{x}=(x_1,\ldots,x_n)$ lies in $P_G(k_1,\ldots,k_m)$ if and only if every point of the form $(\pm x_1,\ldots,\pm x_n)$ lies in $Q_G(k_1,\ldots,k_m)$. By Proposition \ref{prop: Q_G-first-orthant}, we consequently have that part of $Q_G(k_1,\ldots,k_m)$ that lies in each orthant of $\R^n$ is a reflection of $P_G(k_1,\ldots,k_m)$ across some hyperplanes $x_i=0$. The examples given in Figure \ref{fig:graphs-simplices} and Figure \ref{fig:cross-polytopes} help illustrate this symmetry. 
\end{remark}

The following two propositions employ the symmetry of $Q_G(k_1, \dots, k_m)$ to deduce its inequality description and volume formula.

\begin{proposition}\label{prop:inequality-Q_G}
Suppose that $k_1,\ldots,k_m$ are nonnegative real numbers. Then, the polytope $Q_G(k_1,\ldots,k_m)$ is given by the set of all $\vct{x}\in\R^n$ such that, for every non-empty subset $J\subset[n]$,
\[
\sum_{j\in J}|x_j|\le \sum_{I_r\cap J\ne\varnothing}k_r.
\]
\end{proposition}
\begin{proof}
This follows from Lemma \ref{lem:inequalities-P_G}, Proposition \ref{prop: Q_G-first-orthant}, and Remark \ref{rem: Q_G-P_G}.
\end{proof}

\begin{proposition}\label{prop:volume-Q_G}
Let $k_1,\ldots,k_m$ be nonnegative real numbers. Suppose that $Q_G(k_1,\ldots,k_m)$ is $d$-dimensional. Then, the volume of $Q_G(k_1,\ldots,k_m)$ is given by
\[
2^d\sum_{\substack{\vct{a}\in D(G)\\|\vct{a}|=d}}
\frac{k_1^{a_1}\cdots k_m^{a_m}}{a_1!\cdots a_m!},
\]
where $|\vct{a}| = a_1+\cdots+a_m$ is the sum of the components of $\vct{a}$.
\end{proposition}
\begin{proof}
Let $U := I_1\cup \cdots \cup I_m.$ By Proposition \ref{prop:dim-of-Q_G}, we have $d = \dim(Q_G(k_1, \dots, k_m)) = |U|$. We define $G'$ to be the induced subgraph of $G$ obtained by removing every right vertex of degree zero of $G$. By Proposition \ref{prop:dim-of-Q_G}, $Q_{G'}(k_1, \dots, k_m)$ is a polytope of dimension $d$ in $\R^{d}$. Since every point $\vct{x} \in Q_G(k_1, \dots, k_m)$ satisfies $x_i = 0$ for all $i \in [n]\setminus U,$ one sees that $Q_G(k_1, \dots, k_m)$ and $Q_{G'}(k_1, \dots, k_m)$ have the same volume. By Proposition \ref{prop: Q_G-first-orthant}, the polytope $P_{G'}(k_1, \dots, k_m) = Q_{G'}(k_1, \dots, k_m)\cap \R^{d}_{\geq 0}$ is also full dimensional in $\R^{d}$, i.e., $d$-dimensional. The symmetry of $Q_{G'}(k_1, \dots, k_m)$ as described in Remark \ref{rem: Q_G-P_G} and Lemma \ref{lem: volume-P_G} implies that
\begin{align*}\label{eq:volume-calculation}
   \Vol(Q_G(k_1, \dots, k_m)) &= \Vol(Q_{G'}(k_1,\dots, k_m))\\ &= 2^d\Vol(P_{G'}(k_1,\dots, k_m))\\
   &= 2^d\sum_{\substack{\vct{a}\in D(G')\\|\vct{a}|=d}}
\frac{k_1^{a_1}\cdots k_m^{a_m}}{a_1!\cdots a_m!}
\end{align*}
The desired formula given in Proposition \ref{prop:volume-Q_G} follows by noticing that $D(G) = D(G').$
\end{proof}

\section{Lattice-Point Enumeration}\label{sec:lttice-point-enumeration}

In this section, we restrict our attention to those $Q_G(k_1,\ldots,k_m)$ that are lattice polytopes and study their lattice-point enumeration and related consequences.

\subsection{Counting by the support-enumerator}

For $\vct{a}\in\Z_{\ge0}^m$, we write
\[
|\vct{a}|:=a_1+\cdots+a_m,\qquad
\supp(\vct{a}):=\{i\mid a_i>0\},\qquad
s(\vct{a}):=|\supp(\vct{a})|.
\]
Recall the definition of the support-enumerator $F_G(z)$ of $D(G)$ from \eqref{eq:FG}. Since $|\vct{a}|\le n$ for every $\vct{a}\in D(G)$, one sees that $F_G(z)$ is a polynomial of degree at most $n$.

\begin{lemma}\label{lem:Q_G-F_G}
We have
\begin{align}
\label{eq:sign-fibers}    |Q_G(1, \dots, 1)\cap\Z^n|
= F_{G^*}(2)
\end{align}
\end{lemma}
\begin{proof}
By Lemma \ref{lem:inequalities-P_G} and Proposition \ref{prop:inequality-Q_G}, we have $\vct{x} = (x_1, \dots, x_n) \in Q_G(1, \dots, 1)$ if and only if $(|x_1|, \dots, |x_n|) \in P_G(1, \dots, 1)$. Now fix $\vct{a}\in P_G(1, \dots, 1)\cap\Z^n$. The lattice points $\vct{x}\in Q_G(1,\dots,1)\cap\Z^n$ satisfying $(|x_1|, \dots, |x_n|)=\vct{a}$ are obtained by assigning an independent sign to every positive coordinate of $\vct{a}$. A zero coordinate has only one possible sign, while each nonzero coordinate has two. Hence, there are exactly $2^{s(\vct{a})}$ such points $\vct{x}$. Summing over all $\vct{a}\in P_G(1,\dots, 1)\cap\Z^n$ gives
\begin{align}\label{eq:Q_G-F_G}
    |Q_G(1, \dots, 1)\cap\Z^n|
=\sum_{\vct{a}\in P_G(1, \dots, 1)\cap\Z^n}2^{s(\vct{a})}. 
\end{align}
By Lemma \ref{lem: draconion-sequences-dual}, we have $P_G(1,\dots, 1)\cap\Z^n = D(G^*).$ Thus, we see that \eqref{eq:sign-fibers} follows immediately from \eqref{eq:Q_G-F_G}.
\end{proof}

We note that Lemma \ref{lem:Q_G-F_G} implies $|Q_{G^*}(1, \dots, 1)\cap\Z^m|= F_{G}(2).$ By viewing $Q_{G^*}(1, \dots, 1)$ as the "dual" polytope of $Q_{G}(1, \dots, 1)$ and vice versa, we aim to establish that this duality preserves the number of lattice points.

\subsection{The support-enumerator is an $h^*$-polynomial}\label{sec:h-polynomial}

We now show that the support-enumerator defined in \eqref{eq:FG} is the $h^*$-polynomial of a root polytope. As a consequence, this result provides a simple method for computing the $h^*$-polynomial of the corresponding root polytope and, consequently, its volume and Ehrhart polynomial. Moreover, it reveals connections among several seemingly distinct combinatorial objects.

Let $\vct{e}_1,\ldots,\vct{e}_m$ be the standard basis of $\R^{m}$ and let $\vct{f}_1,\ldots,\vct{f}_n$ be the standard basis of $\R^{n}$. Given a bipartite graph $G=(L\sqcup R,E)$ on $m$ left and $n$ right vertices, the \emph{root polytope} of $G$ is
\begin{equation}
\mathcal R_{G}
:=\conv(\vct{e}_i+\vct{f}_j\mid (\ell_i,r_j)\in E)
\subset\R^{m}\oplus\R^{n}.
\label{eq:root-polytope}
\end{equation}

The \emph{augmented bipartite graph} of $G$ is the bipartite graph 
\[
\widehat G=(\widehat L\sqcup\widehat R,\widehat E)
\]
where $\widehat L:=\{\ell_0\}\sqcup L, 
\widehat R:=\{r_0\}\sqcup R,$ and $ 
\widehat E
:=E\cup\{(\ell_0,v)\mid v\in\widehat R\}
\cup\{(u,r_0)\mid u\in L\}.$
That is, $\widehat G$ is obtained from $G$ by adding the left vertex $\ell_0$ and the right vertex $r_0$, and letting $\ell_0$ be adjacent to every right vertex, and $r_0$ be adjacent to every left vertex. Thus, $\widehat G$ is connected.

By definition, we have that $R_{\widehat G}$ is a polytope in $\R^{m+1}\oplus \R^{n+1}$. For every nonnegative integer $k$, a point $\vct{p}\in k\mathcal R_{\widehat G}$ can be represented in the form
\begin{equation}
\vct{p}
=\sum_{(\ell_i,r_j)\in\widehat E}\lambda_{ij}(\vct{e}_i+\vct{f}_j),
\text{ for some }
\lambda_{ij}\ge0 \text{ such that }
\sum_{(\ell_i,r_j)\in\widehat E}\lambda_{ij}=k.
\label{eq:flow}
\end{equation}
The coefficients $\lambda_{ij}$ can be regarded as a nonnegative \emph{weight} of the edge $(\ell_i,r_j) \in \widehat E$. By writing $\vct{p} = (\alpha_0,\dots, \alpha_m,\beta_0,\dots, \beta_n)$, we see that the left coordinates
\[
\alpha_i:=\sum_{j=0}^n\lambda_{ij}
\]
are the \emph{$i^{\mathrm{th}}$-row-sums} of this weight, and the right coordinates
\[
\beta_j:=\sum_{i=0}^m\lambda_{ij}
\]
are the \emph{$j^{\mathrm{th}}$-column-sums}. Thus, a point in the root polytope $R_{\widehat G}$ records possible pairs of \emph{left marginals} $(\alpha_0,\dots, \alpha_m)$ and \emph{right marginals} $(\beta_0,\dots,\beta_n)$ of the weights on the edges of $\widehat G$.

\begin{lemma}\label{lem:dimension-root-polytope}
If $G$ is a connected bipartite graph on $m$ left vertices and $n$ right vertices, then
\[
\dim\mathcal R_G=m+n-2.
\]
In particular,
\[
\dim\mathcal R_{\widehat G}=m+n.
\]
\end{lemma}
\begin{proof}
Every point $\vct{e}_i+\vct{f}_j$ of $\mathcal R_G$ lies in the two hyperplanes
\[
\sum^m_{i = 1} x_i=1,
\qquad
\sum_{j=1}^n y_j=1.
\]
Since these two hyperplanes are distinct in $\R^{m}\oplus \R^n$, it follows that $\dim\mathcal R_G\le m+n-2$.

For the reverse inequality, choose a spanning tree $T$ of $G$. The tree has $m+n-1$ edges. We claim that the $m+n-1$ vectors
\[
\{\vct{e}_i+\vct{f}_j\mid (\ell_i,r_j)\in E(T)\}
\]
are affinely independent. To see this, suppose that
\[
\sum_{(\ell_i,r_j)\in E(T)}c_{ij}(\vct{e}_i+\vct{f}_j)=\vct{0}.
\]
Choose a leaf of $T$. Since the coordinate belonging to this leaf occurs in exactly one vector in the sum, the coefficient of the unique incident edge must be zero. Remove this leaf and its edge. By repeating this process, we see by induction that $c_{ij}=0$ for all $(\ell_i,r_j) \in E(T)$. Hence, the $m+n-1$ vectors are linearly, and therefore affinely, independent. Consequently, their convex hull has dimension $m+n-2$, which then implies $\dim\mathcal R_G\ge m+n-2$. Thus, the equality follows.

The graph $\widehat G$ has $(m+1)+(n+1)=m+n+2$ vertices and is connected, hence its root polytope has dimension $m+n$.
\end{proof}


Let $G^+$ denote the induced subgraph of $\widehat G$ obtained by removing the right vertex $r_0$ of $\widehat G.$ The next lemma establishes a key connection between the root polytope $R_{\widehat G}$ and a Minkowski sum of simplices.

\begin{lemma}\label{lem:slicing-root-polytope}
Fix $\vct{\alpha}=(\alpha_0,\ldots,\alpha_m)\in\Z_{\ge0}^{m+1}$ and set $
k:=\alpha_0+\cdots+\alpha_m.$ Let 
\[H_{\vct{\alpha}} :=\{(x_0,\dots, x_m,y_0,\dots, y_n) \in \R^{m+1}\oplus \R^{n+1}\mid (x_0,\dots,x_m) = \vct{\alpha}\}\]
be an affine subspace in $\R^{m+1}\oplus \R^{n+1}.$
Then, the polytope $P_{\widehat G,\vct{\alpha}}:=kR_{\widehat G}\cap H_{\vct{\alpha}}$ is integrally equivalent to $P_{G^+}(\alpha_0,\dots, \alpha_m)$, and consequently
\begin{equation}
|k\mathcal R_{\widehat G}\cap\Z^{m+n+2}|
=\sum_{\substack{\alpha_i\ge0\\\alpha_0+\cdots+\alpha_m=k}}
|P_{G^+}(\alpha_0,\ldots,\alpha_m)\cap\Z^n|.
\label{eq:slicing}
\end{equation}
\end{lemma}
\begin{proof} Let
\[
\pi:\R^{m+1}\oplus\R^{n+1}\longrightarrow\R^{n}
\]
be the projection given by $\pi(x_0,\dots, x_m,y_0,\dots, y_n) := (y_1\dots, y_n)$. We first show that $\pi$ defines an invertible transformation from $\aff(P_{\widehat G,\vct{\alpha}})$ to $\aff(P_{G^+}(\alpha_0,\dots,\alpha_m))$ and preserves the lattice points between $P_{\widehat G,\vct{\alpha}}$ and $P_{G^+}(\alpha_0,\dots,\alpha_m)$. 

Note that every point $\vct{p}$ in $P_{\widehat G,\vct{\alpha}}$
 has the form $\vct{p}=(\alpha_0,\dots,\alpha_m,\beta_0,\dots, \beta_n)$
 for some nonnegative real numbers $\beta_0,\dots, \beta_n.$ By representing $\vct{p}$ as in \eqref{eq:flow}, we define for each $i\in\{0,1,\ldots,m\}$
\[
\vct{w}^{(i)}:=(\lambda_{i1},\ldots,\lambda_{in})\in\R_{\ge0}^n.
\]
Since the $i^{\mathrm{th}}$-row-sum is $\alpha_i$, we have
\[
\sum_{j=1}^n w_j^{(i)}\le\alpha_i.
\]
Therefore,
$\vct{w}^{(0)}\in\alpha_0\Delta_{[n]}^0$ and $\vct{w}^{(i)}\in\alpha_i\Delta_{I_i}^0$ for all $i\in [m]$. Since
\[
\vct{\beta}:=(\beta_1,\ldots,\beta_n)
=\sum_{i=0}^m\vct{w}^{(i)},
\]
it follows that
\[
\vct{\beta}
\in\alpha_0\Delta_{[n]}^0+\sum_{i=1}^m\alpha_i\Delta_{I_i}^0
=P_{G^+}(\alpha_0,\ldots,\alpha_m).
\]
This shows that the projection $\pi$ satisfies $\pi(\vct{p}) = \vct{\beta} \in P_{G^+}(\alpha_0,\dots, \alpha_m)$ for all $\vct{p} \in P_{\widehat G,\vct{\alpha}}$, and hence defines a map from $P_{\widehat G,\vct{\alpha}}$ to $P_{G^+}(\alpha_0,\dots, \alpha_m).$ 

Conversely, suppose
$\vct{\beta}\in P_{G^+}(\alpha_0,\ldots,\alpha_m).$ By the definition of a Minkowski sum, we may choose $\vct{w}^{(0)}\in\alpha_0\Delta_{[n]}^0$ and $
\vct{w}^{(i)}\in\alpha_i\Delta_{I_i}^0$
such that $\vct{\beta}=\sum^m_{i=0}\vct{w}^{(i)}$. Set
\[
\lambda_{ij}:=w_j^{(i)} \text{ for }j\ge1
\quad \text{and } \quad
\lambda_{i0}:=\alpha_i-\sum_{j=1}^n w_j^{(i)}\ge0.
\]

Hence, $\lambda_{ij}$ can be regarded as a nonnegative weight of the edge $(i,j)$ of $\widehat G$. All these weights of the edges of $\widehat G$ have their $i^{\mathrm{th}}$-row-sums equal to $\alpha_i$ for all $i \in [m]$. By letting
\[
\beta_0=k-\sum_{j=1}^n\beta_j,
\]
one sees that $\vct{p}:=(\vct{\alpha},\beta_0,\vct{\beta})$ is a point in $P_{\widehat G, \vct{\alpha}}$ satisfying $\pi(\vct{p}) = \vct{\beta}.$  Clearly, once $\vct{\alpha}$ and $\vct{\beta}$ are fixed, such a point $\vct{p}$ is uniquely determined. Moreover, $\vct{\beta}$ is a lattice point if and only if $\vct{p}$ is a lattice point. This shows that $\pi$ defines an invertible map from  $P_{\widehat G,\vct{\alpha}}$ to $P_{G^+}(\alpha_0,\dots, \alpha_m)$ and preserves their lattice points. The linearity of $\pi$ then implies that it actually defines an invertible affine transformation from $\aff(P_{\widehat G,\vct{\alpha}})$ to $\aff(P_{G^+}(\alpha_0,\dots,\alpha_m))$. Therefore, the two polytopes $P_{\widehat G,\vct{\alpha}}$ and $P_{G^+}(\alpha_0,\dots,\alpha_m)$ are integrally equivalent.

Finally, every lattice point of $k\mathcal R_{\widehat G}$ has a unique integral left marginal $\vct{\alpha}\in\Z_{\ge0}^{m+1}$ with $|\vct{\alpha}|=k$. Thus, summing over their cardinalities yields \eqref{eq:slicing}.
\end{proof}

As a consequence of Lemma \ref{lem:slicing-root-polytope}, we obtain a realization of the support-enumerator $F_G(z)$ as the $h^*$-polynomial of the root polytope $R_{\widehat G}$, as stated in Theorem \ref{thm:F_G-h-polynomial}. We provide a proof below.
\begin{proof}[Proof of Theorem \ref{thm:F_G-h-polynomial}]
Using the identity \eqref{eq:slicing} and Lemma \ref{lem:lattice-points-formula-P_G}, we express the Ehrhart series of $R_{\widehat G}$ as
\begin{align*}
\Ehr_{\mathcal R_{\widehat G}}(z)
&=\sum_{\vct{\alpha}\in\Z_{\ge0}^{m+1}}
|P_{G^+}(\alpha_0,\ldots,\alpha_m)\cap\Z^n|z^{|\vct{\alpha}|}\\
&=\sum_{\vct{\alpha}\in\Z_{\ge0}^{m+1}}
\sum_{\vct{a}\in D(G^+)}
\prod_{i=0}^m\binom{\alpha_i+a_i-1}{a_i}z^{\alpha_i}\\
&=\sum_{\vct{a}\in D(G^+)}
\prod_{i=0}^m\left(
\sum_{\alpha_i\ge0}\binom{\alpha_i+a_i-1}{a_i}z^{\alpha_i}
\right).
\end{align*}
For a nonnegative integer $a$, we define
\[
\Phi_a(z):=\sum_{\alpha\ge0}\binom{\alpha+a-1}{a}z^\alpha.
\]
The binomial-series identity gives
\begin{equation}
\Phi_a(z)=
\begin{cases}
\dfrac{1}{1-z},&a=0,\\[4pt]
\dfrac{z}{(1-z)^{a+1}},&a>0.
\end{cases}
\label{eq:Phi}
\end{equation}
Therefore,
\begin{equation}
\Ehr_{\mathcal R_{\widehat G}}(z)
=\sum_{\vct{a}\in D(G^+)}
\frac{z^{s(\vct{a})}}{(1-z)^{m+1+|\vct{a}|}}.
\label{eq:Ehr-Gplus}
\end{equation}

We now describe $D(G^+)$. Write $\vct{a}=(a_0,\vct{b})$ with $\vct{b}=(b_1,\ldots,b_m)$. The inequalities in the description of draconian sequences corresponding to subsets not containing $0$ imply $\vct{b}\in D(G)$. Since the neighborhood of $\ell_0$ in $G^+$ is $I_0 = [n]$, we have
\[a_0 \leq n \quad \text{and} \quad
a_0+|\vct{b}|\le n.
\]
Hence,
\begin{equation}
D(G^+)=\{(a_0,\vct{b})\mid \vct{b}\in D(G),\ 0\le a_0\le n-|\vct{b}|\}.
\label{eq:D-Gplus}
\end{equation}
Substituting \eqref{eq:D-Gplus} into \eqref{eq:Ehr-Gplus}, and setting $N_{\vct{b}}:=n-|\vct{b}|$ gives
\[
\Ehr_{\mathcal R_{\widehat G}}(z)
=\sum_{\vct{b}\in D(G)}\sum_{a_0=0}^{N_{\vct{b}}}
\frac{z^{s(\vct{b})+\mathbf 1_{a_0>0}}}
{(1-z)^{m+1+|\vct{b}|+a_0}}.
\]
Bringing all terms for a fixed $\vct{b}$ to the common denominator $(1-z)^{m+n+1}$, we see that the corresponding numerator is
\begin{align*}
z^{s(\vct{b})}
\left[(1-z)^{N_{\vct{b}}}+z\sum_{a_0=1}^{N_{\vct{b}}}(1-z)^{N_{\vct{b}}-a_0}\right]
&=z^{s(\vct{b})}
\left[(1-z)^{N_{\vct{b}}}+z\sum_{r=0}^{N_{\vct{b}}-1}(1-z)^r\right]\\
&=z^{s(\vct{b})}
\left[(1-z)^{N_{\vct{b}}}+1-(1-z)^{N_{\vct{b}}}\right]\\
&=z^{s(\vct{b})}.
\end{align*}
The same computation is valid when $N_{\vct{b}}=0$, in which case the finite sum is empty. Summing over $\vct{b}$ yields
\[
\Ehr_{\mathcal R_{\widehat G}}(z)
=\frac{\sum_{\vct{b}\in D(G)}z^{s(\vct{b})}}{(1-z)^{m+n+1}}
=\frac{F_G(z)}{(1-z)^{m+n+1}}.
\]
By Lemma \ref{lem:dimension-root-polytope}, $\dim\mathcal R_{\widehat G}=m+n$. Recall that the Ehrhart series of a $d$-dimensional lattice polytope $P$ is given by
\[
\Ehr_P(z)=\frac{h_P^*(z)}{(1-z)^{d+1}}.
\]
Since the denominator in \eqref{eq:root-ehrhart} is $(1-z)^{(m+n)+1}$, the uniqueness of the numerator proves \eqref{eq:hstar-F}.
\end{proof}

In \cite{KalmanPostnikov2017}, K\'alm\'an and Postnikov study the $h^*$-polynomials of root polytopes associated to bipartite graphs and show that they coincide with \emph{interior polynomials} of corresponding \emph{hypergraphs}. Later in \cite{Kalman2013RootPP}, K\'alm\'an and T\'othm\'er\'esz provide a method for computing the coefficients of these interior polynomials in terms of \emph{activities} of \emph{dissecting spanning trees} of a graph. Thus, our result provides a new interpretation and simplifies the computation of the $h^*$-polynomials of these root polytopes.  

We note that Theorem \ref{thm:F_G-h-polynomial} can also be proved using the hypergraph framework introduced  in \cite{KalmanPostnikov2017} as follows. View the left vertices $\ell_0,\ell_1,\dots,\ell_m$ as \emph{hyperedges} and order them in this order. A \emph{hypertree} $(a_0,a_1,\dots, a_m)$ is then a sequence such that $(a_1,\dots, a_m)$ is a $G$-draconian sequence, and $a_0 + a_1 + \cdots + a_m = n.$ Under this correspondence, the number of \emph{internally inactive} hyperedges is precisely $s(a_1,\dots, a_m)$, the number of supports of $(a_1,\dots, a_m).$ Thus, the correspondence between hypertrees and draconian sequences directly identifies the interior polynomial with the support-enumerator.

Interpreting these enumerators as $h^*$-polynomials reveals their connections to a broad range of combinatorial objects. Note that the interior polynomial of a hypergraph generalizes the specialization $T(x,1)$ of the Tutte polynomial \cite{Tutte1954}, a fundamental invariant that encodes a wealth of combinatorial information about graphs; see, for example, the survey \cite{Ellis-nonaghanMerino2011}. When the augmented graph $\widehat G$ is \textit{planar}, meaning that it admits a drawing in $\mathbb{R}^2$ in which no two edges cross, the corresponding interior polynomial is also related to the \textit{HOMFLY polynomials} of \textit{special alternating links} \cite{Kalman2013RootPP} and coincides with the \textit{parking function enumerator} of $\widehat G^*$ \cite[Corollary 5.9]{KalmanPostnikov2017} studied in \cite{PostnikovBoris2004}. Moreover, recent work connects root polytopes to toric geometry \cite{Rietsch_Williams_2025} and toric edge ideals \cite{AlmousaEtal2025}. Hence, support-enumerators provide a bridge linking these different combinatorial and algebraic objects, raising the question of what information about them is encoded by the support-enumerators.

Recall from Section \ref{sec:Ehrhart-theory} that the $h^*$-polynomial $h^*_P(z)$ of a lattice polytope $P$ can be used for computing its Ehrhart polynomial as shown in \eqref{eq:h-polynomial-Ehrhart-Polynomial}, and the sum of the coefficients of $h^*_P(z)$ equals $\NVol(P)$. Hence, we obtain the following simple calculation of $\NVol(R_{\widehat G}).$

\begin{corollary}\label{cor:normalized-volume-root-polytope}
    The normalized volume of the root polytope $R_{\widehat G}$ is given by 
    $$\NVol(R_{\widehat G}) = F_G(1) = |D(G)|.$$
\end{corollary}
\begin{proof}
    The first equality follows from the fact that the sum of the coefficients of $h^*_{R_{\widehat G}}(z) = F_G(z)$ equals $\NVol(R_{\widehat G})$. The second equality follows directly from the definition of $F_G(z)$ in \eqref{eq:FG}.
\end{proof}

\subsection{Duality}

By exploiting the symmetry of the root polytope $R_{\widehat G}$, we obtain a proof of the duality property for support-enumerators stated in Theorem \ref{thm:duality-support-enumerator}.
\begin{proof}[Proof of Theorem \ref{thm:duality-support-enumerator}]
We define the map $T:\R^{m+1}\oplus\R^{n+1}\longrightarrow \R^{n+1}\oplus\R^{m+1}$ by
$$T(x_0,x_1,\ldots,x_m,y_0,y_1,\ldots,y_n) = (y_0,y_1,\ldots,y_n,x_0,x_1,\ldots,x_m).$$ It is clear that $T$ is an invertible transformation. Since $T(\vct{e}_i + \vct{f}_j) = \vct{f}_j + \vct{e}_i \in \R^{n+1}\oplus \R^{m+1}$ for all $i$ and $j$, the map $T$ also defines a bijection from $R_{\widehat G}$ to $R_{\widehat {G^*}}$ and preserves their lattice points. By linearity, the map $T$ defines an invertible affine transformation from $\aff(R_{\widehat {G}})$ to $\aff(R_{\widehat {G^*}})$. Hence, $R_{\widehat {G}}$ is integrally equivalent to $R_{\widehat {G^*}}.$ This implies that $h^*_{\mathcal R_{\widehat G}}(z)
=h^*_{\mathcal R_{\widehat{G^*}}}(z).$ Applying Theorem \ref{thm:F_G-h-polynomial} to $G$ and $G^*$ gives
$F_G(z)=h^*_{\mathcal R_{\widehat G}}(z)
=h^*_{\mathcal R_{\widehat{G^*}}}(z)
=F_{G^*}(z).$
\end{proof}

As mentioned in Section \ref{sec:intro}, Theorem \ref{thm:duality-support-enumerator} proves a conjecture of Athanasiadis and Chapoton in \cite{chapoton2026} concerning the $h$-polynomials of preorders. We briefly review their study of \emph{preorder polytopes} and explain how Theorem \ref{thm:duality-support-enumerator} establishes their conjecture.

A \emph{preorder} $\tau$ on the set $[n]$ can be viewed as a partial order on the set of blocks of a partition of $[n].$ We call these blocks \emph{vertices} of the preorder $\tau$, and write $a \leq b$ if $a$ lies in a vertex $A$ and $b$ lies in a vertex $B$ such that $A$ is less than or equal to $B$. A subset $I \subseteq [n]$ is an \emph{order ideal} of $\tau$ if for every $b \in I$, we have that $a \in I$ for all $a \leq b.$ The \emph{preorder polytope} associated to $\tau$, denoted by $\mathcal{Q}_\tau$, is defined to be the set of all points $\vct{x} \in \R^n$ such that $0 \leq x_i$ for all $i \in [n]$, and for every order ideal $I$ of $\tau$
\[\sum_{i \in I}x_i \leq |I|.\]
The $h$-\textit{polynomial} of $\tau$, denoted by $h(\tau,z)$, is defined by
\[h(\tau,z):= h_0(\tau) + h_1(\tau)z +\cdots +h_d(\tau)z^d\]
where $h_i(\tau)$ is the number of lattice points of $\mathcal{Q}_\tau$ having exactly $i$ positive coordinates. 
\begin{conjecture}[Athanasiadis--Chapoton \cite{chapoton2026}]
For every preorder $\tau$,
\[h(\tau,z) = h(\tau^*,z)\]
where $\tau^*$ is the \emph{dual} preorder of $\tau,$ i.e., $a \leq b$ in $\tau^*$ if and only if $b \leq a$ in $\tau.$
\end{conjecture}

Let $G_\tau$ be the bipartite graph on $n$ left vertices $\ell_1,\dots, \ell_n$ and $n$ right vertices $r_1, \dots, r_n$ such that $(\ell_i,r_j)$ is an edge of $G$ if $i \leq j$ in the preporder $\tau$. One can verify using Lemma \ref{lem:inequalities-P_G} that $\mathcal{Q}_\tau$ is precisely the polytope $P_{G_\tau}(1,\dots, 1).$ Thus, the $h$-polynomial $h(\tau,z)$ coincides with the support-enumerator $F_{G^*_\tau}(z)$ of $D(G^*_\tau)$. Moreover, since $G^*_\tau = G_{\tau^*}$, we have by Theorem \ref{thm:duality-support-enumerator} that
\[h(\tau,z) = F_{G^*_\tau}(z) = F_{G_{\tau^*}}(z) = F_{G^*_{\tau^*}}(z) = h(\tau^*,z),\]
proving the conjecture.

As a further consequence of Theorem \ref{thm:duality-support-enumerator}, we establish the following duality property for Minkowski sums of cross polytopes. 

\begin{corollary}\label{cor:duality-for-Q_G}
For every bipartite graph $G$, we have that
\[
|Q_G(1,\ldots,1)\cap\Z^n|
=|Q_{G^*}(1,\ldots,1)\cap\Z^m|.
\]
More strongly, the two polytopes have the same distribution of the number of nonzero coordinates:
\[
\sum_{\vct{x}\in Q_G\cap\Z^n}u^{s(\vct{x})}
=\sum_{\vct{y}\in Q_{G^*}\cap\Z^m}u^{s(\vct{y})}
\]
where $Q_G := Q_G(1,\dots, 1)$ and $Q_{G^*}:= Q_{G^*}(1,\dots,1).$
\end{corollary}
\begin{proof}
Lemma \ref{lem:Q_G-F_G} and Theorem \ref{thm:duality-support-enumerator} immediately give
\[
|Q_G\cap\Z^n|=F_{G^*}(2) = F_G(2) =
|Q_{G^*}\cap\Z^m|.
\]

For the stronger assertion, every lattice point $\vct{a}\in P_G(1,\dots,1)$ with nonnegative coordinates contributes $2^{s(\vct{a})}$ lattice points $\vct{x} = (x_1,\dots, x_n)$ in $Q_G$ with $(|x_1|,\dots, |x_n|) = \vct{a}$. Moreover, all such points $\vct{x}$ have the same support size $s(\vct{a})$. Therefore, together with Theorem \ref{thm:duality-support-enumerator}, we deduce
\begin{align*}
    \sum_{\vct{x}\in Q_G\cap\Z^n}u^{s(\vct{x})}
=\sum_{\vct{a}\in D(G^*)}(2u)^{s(\vct{a})}
=F_{G^*}(2u)
= F_{G}(2u) = \sum_{\vct{a}\in D(G)}(2u)^{s(\vct{a})} = \sum_{\vct{x}\in Q_{G^*}\cap\Z^m}u^{s(\vct{x})}.
\end{align*}
\end{proof}

Note that Corollary \ref{cor:duality-for-Q_G} establishes duality properties for Minkowski sums of cross polytopes, analogous to those for Minkowski sums of simplices established in Lemma \ref{lem: draconion-sequences-dual} and Theorem \ref{thm:duality-support-enumerator}.

\begin{example}
    By Example \ref{ex:draconian-sequences} and the symmetry of cross polytopes, we have
    \begin{align*}
        Q_{G_1}(1,1,1)\cap \Z^2 &= \{(0,0), (\pm 1,0),(0,\pm 1),(\pm 2,0),(\pm 1,\pm 1),(0,\pm 2),(\pm 2,\pm 1),(\pm 1,\pm 2)\},\\
        Q_{G_1^*}(1,1)\cap \Z^3&=Q_{G_2}(1,1)\cap \Z^3 = \{(0,0,0), (\pm 1,0,0), (0,\pm 1,0),(0,0,\pm 1),(\pm 1,\pm 1,0),\\
        & \qquad \qquad \qquad\qquad\qquad (\pm 1,0,\pm 1),(0,\pm 1,\pm 1),(0,\pm 2,0)\},\\
        F_{G_1}(z) &= 1+4z+3z^2= F_{G^*_1}(z) = F_{G_2}(z).
    \end{align*}
Thus, we see that $|Q_{G_1}(1,\ldots,1)\cap\Z^n|
= 21 = F_{G_1}(2)= |Q_{G^*_1}(1,\ldots,1)\cap\Z^m| ,$ and
\[\sum_{\vct{x}\in Q_{G_1}\cap\Z^n}u^{s(\vct{x})}
= 1+8u+12u^2  =F_{G_1}(2u) = \sum_{\vct{x}\in Q_{G^*_1}\cap\Z^m}u^{s(\vct{x})}.
\]
\end{example}

\subsection{A formula for counting lattice points}

We fist establish a few auxiliary results and then use them to prove the formula for the number of lattice points in Theorem \ref{thm:general-formula-for-lattice-points}.

\begin{definition}
For $\vct{k}=(k_1,\ldots,k_m)\in\Z_{\ge0}^m$, we define the \emph{clone graph} $G^{(\vct{k})}$ of $G$ by replacing the left vertex $\ell_i$ with $k_i$ \emph{clones}
\[
\ell_{i,1},\ldots,\ell_{i,k_i},
\]
each having the same neighborhood $I_i$. If $k_i=0$, no clone of $\ell_i$ is present.
\end{definition}
By construction, the polytope $Q_{G^{(\vct{k})}}(1,\dots, 1)$ is simply a way to view $Q_{G}(k_1,\dots, k_m)$ as 
\[\underbrace{\diamondsuit_{I_1} + \cdots +\diamondsuit_{I_1}}_{k_1 \text{ terms }} + \cdots + \underbrace{\diamondsuit_{I_m} + \cdots + \diamondsuit_{I_m}}_{k_m \text{ terms }}.\]
Thus,
\begin{equation}
Q_{G^{(\vct{k})}}(1,\ldots,1)
=Q_G(k_1,\ldots,k_m).
\label{eq:clone-Q}
\end{equation}

\begin{lemma}\label{lem:group-sums-draconion-sequences} For $\vct{k}=(k_1,\ldots,k_m)\in\Z_{\ge0}^m$, let
\[
\vct{b}=(b_{i,s})_{1\le i\le m,\ 1\le s\le k_i}
\in\Z_{\ge0}^{k_1+\cdots+k_m},
\]
and define its \emph{group-sums}
\[
a_i:=\sum_{s=1}^{k_i}b_{i,s}.
\]
Then, we have that $\vct{b}\in D(G^{(\vct{k})})
$ if and only if $\vct{a}=(a_1,\ldots,a_m)\in D(G).$
\end{lemma}
\begin{proof}
Assume first that $\vct{b}\in D(G^{(\vct{k})})$. For a subset $S\subset[m]$, take all clones belonging to groups indexed by $S$. Then, the sum of their corresponding draconian sequences is $\sum_{i\in S}a_i$, and the union of their neighbors is $\bigcup_{i\in S}I_i$. The inequality description of the draconian sequences gives
\[
\sum_{i\in S}a_i\le\left|\bigcup_{i\in S}I_i\right|.
\]
Hence, $\vct{a}\in D(G)$.

Conversely, assume $\vct{a}\in D(G)$ and let $U$ be any subset of the clones. Let
\[
S:=\{i\mid U\text{ contains at least one clone of }\ell_i\}.
\]
Then
\begin{align}\label{eq:clone-union}
    \sum_{(i,s)\in U}b_{i,s}
\le\sum_{i\in S}a_i
\le\left|\bigcup_{i\in S}I_i\right|.
\end{align}
The union in \eqref{eq:clone-union} is exactly the union of the neighbors of the clones in $U$, since all clones of a vertex have the same neighbors. Thus, we see that $\vct{b}$ satisfies the inequality description of the $G^{(\vct{k})}$-draconian sequences. Hence, $\vct{b} \in D(G^{(\vct{k})}).$
\end{proof}

For nonnegative integers $a$ and $k,$ we define
\begin{equation}
c(a,k):=\#(\vct{x}\in\Z^k \text{ such that } |x_1|+\cdots+|x_k|=a),
\label{eq:c-def}
\end{equation}
and set $c(0,0)=1$, and $c(a,0)=0$ for $a>0$.

\begin{lemma}\label{lem:explicit-c(a,k)} Let $a,k$ be nonnegative integers. Then, we have
\[
c(0,k)=1,
\]
and, for $a\ge1$,
\begin{equation}
c(a,k)=\sum_{r=1}^{\min(a,k)}2^r\binom{k}{r}\binom{a-1}{r-1}.
\label{eq:c-explicit}
\end{equation}
Equivalently,
\begin{equation}
c(a,k)=\sum_{\substack{b_1,\ldots,b_k\ge0\\b_1+\cdots+b_k=a}}2^{s(\vct{b})}.
\label{eq:c-compositions}
\end{equation}
\end{lemma}
\begin{proof}
For \eqref{eq:c-explicit}, choose the number $r$ of nonzero coordinates, choose their locations in $\binom{k}{r}$ ways, write $a$ as an ordered sum of $r$ positive integers in $\binom{a-1}{r-1}$ ways, and choose the $r$ signs in $2^r$ ways.

For \eqref{eq:c-compositions}, fix the vector $\vct{b} = (b_1,\dots, b_k)$ of nonnegative coordinates satisfying $b_1+\cdots + b_k = a.$ Since $\vct{b}$ has $s(\vct{b})$ positive coordinates, exactly $2^{s(\vct{b})}$ integer vectors $\vct{x}$ satisfy $\vct{b} = (|x_1|,\dots, |x_k|)$. Summing over all weak compositions of $a$ proves the identity.
\end{proof}

\begin{example}
By viewing $k$ as an indeterminate, the function $c(a,k)$ is a polynomial in $k$. We list them for $a \in [6]$ here.
\begin{align*}
c(1,k) &= 2k, &
c(2,k) &= 2k^2,\\
c(3,k) &= \frac{4}{3}k^3+\frac{2}{3}k, & 
c(4,k) &= \frac{2}{3}k^4+\frac{4}{3}k^2,\\
c(5,k) &= \frac{4}{15}k^5+\frac{4}{3}k^3+\frac{2}{5}k, &
c(6,k) &= \frac{4}{45}k^6+\frac{8}{9}k^4+\frac{46}{45}k^2.
\end{align*}

 Notice that the coefficients of $c(a,k)$ are nonnegative and that $c(a,k)$ only contains the terms $k^r$ for $r$ having the same parity as $a.$ In Lemma \ref{lem:properties-of-c(a,x)}, we show that these observations hold for all nonnegative integers $a.$
\end{example}

\begin{proof}[Proof of Theorem \ref{thm:general-formula-for-lattice-points}]
Let $H:=G^{(\vct{k})}$. By \eqref{eq:clone-Q}, Lemma \ref{lem:Q_G-F_G}, and Corollary \ref{cor:duality-for-Q_G}, we have
\[
|Q_G(k_1,\ldots,k_m)\cap\Z^n|
=|Q_H(1,\ldots,1)\cap\Z^n|
=F_H(2)
=\sum_{\vct{b}\in D(H)}2^{s(\vct{b})}.
\]
Group the vectors $\vct{b}$ according to their group-sum vectors $\vct{a}$. By Lemma \ref{lem:group-sums-draconion-sequences}, the possible group-sum vectors are exactly the vectors $\vct{a}\in D(G)$. For a fixed $\vct{a}$, the choices in distinct clone groups are independent, and Lemma \ref{lem:explicit-c(a,k)} gives
\begin{align*}
\sum_{\substack{\vct{b}\in D(H)\\\sum_s b_{i,s}=a_i\ \forall i}}2^{s(\vct{b})}
&=\prod_{i=1}^m
\left(
\sum_{\substack{b_{i,1},\ldots,b_{i,k_i}\ge0\\\sum_s b_{i,s}=a_i}}
2^{s(b_{i,1},\ldots,b_{i,k_i})}
\right)\\
&=\prod_{i=1}^m c(a_i,k_i).
\end{align*}
Summing over $\vct{a}\in D(G)$ gives \eqref{eq:general-count}.
\end{proof}

\begin{remark}\label{rem:validity-of-negative-coefficients-Q_G}
    Similarly to the property of $P_G(y_1,\dots,y_m)$ described in Remark \ref{rem:validity-of-negative-coefficients-P_G}, when some of the integers $k_1, \dots, k_m$ are negative, the formula \eqref{eq:general-count} in Theorem \ref{thm:general-formula-for-lattice-points} may still be valid for $Q_G(k_1, \dots, k_m)$. More precisely, due to the symmetry of $Q_G(k_1,\dots, k_m)$, whenever $P_G(k_1,\dots, k_m)$ satisfies \eqref{eq:Minkowski-summand-P_G}, we also have that formula \eqref{eq:general-count} still holds for $Q_G(k_1,\dots, k_m)$.
\end{remark}

\subsection{Ehrhart Positivity}

Since scalar multiplication distributes over Minkowski sums, we have
\[
tQ_G(k_1, \dots, k_m)=Q_G(tk_1,\ldots,tk_m)
\]
for all $t \in \Z_{\geq 0}.$ Thus, Theorem~\ref{thm:general-formula-for-lattice-points} immediately gives its Ehrhart polynomial.

\begin{corollary} The Ehrhart polynomial of $Q_G(k_1, \dots, k_m)$ is given by
\begin{equation}
\sum_{\vct{a}\in D(G)}\prod_{i=1}^m c(a_i,tk_i)
\label{eq:ehrhart-formula}
\end{equation}
for all $t \in \Z_{\geq 0}.$
\end{corollary}

Next, we establish several properties of the function $c(a,k)$ that are key to proving the Ehrhart positivity of polytopes in this family.

\begin{lemma}\label{lem:properties-of-c(a,x)}
For every fixed integer $a\ge0$, the function $c(a,k)$ is a polynomial in $k$, $k \in \Z_{\geq 0}$, of degree $a$ satisfying the following properties.
\begin{enumerate}
\item\label{itm:nonnegative-coefficients-c(a,x)} Every coefficient of $c(a,k)$ is a nonnegative rational number.
\item\label{itm:leading-coefficient-c(a,x)} Its leading coefficient is $2^a/a!$.
\item\label{itm:parity-of-c(a,x)} The coefficient of $k^r$ is zero for all nonnegative integers $r$ with the same parity as $a+1$.
\end{enumerate}
\end{lemma}

Note that property \ref{itm:parity-of-c(a,x)} in Lemma \ref{lem:properties-of-c(a,x)} implies that the polynomial  $c(a,k)$ is an odd function when $a$ is odd, and is an even function when $a$ is even. That is, we have $c(a,-k) = (-1)^ac(a,k)$ for all $a \in \Z_{\geq 0}$ and $k \in \Z_{\geq 0}.$

\begin{proof} For a fixed nonnegative integer $a,$ the formula \eqref{eq:c-explicit} in Lemma \ref{lem:explicit-c(a,k)} implies that $c(a,k)$ is a polynomial of degree $a$ in variable $k$ with rational coefficients and leading coefficient $2^a/a!$.

For an integer $k \geq 0$, it is not difficult to see from the definition of $c(a,k)$ in \eqref{eq:c-def} that 
\begin{equation}
\sum_{a\ge0}c(a,k)u^a
=\left(\frac{1+u}{1-u}\right)^k
\label{eq:d-generating}
\end{equation}
as each component of $\vct{x}$ with $|x_1| + \cdots |x_k| = a$ contributes the generating function \[1+2u+2u^2+\cdots=\frac{1+u}{1-u}.\]
We now treat $k$ as an indeterminate and write
\begin{align*}
\left(\frac{1+u}{1-u}\right)^k
&=\exp\left(k\log\frac{1+u}{1-u}\right)\\
&=\exp\left(2k\sum_{q\ge0}\frac{u^{2q+1}}{2q+1}\right).
\end{align*}
By setting
\[
B(u):=2\sum_{q\ge0}\frac{u^{2q+1}}{2q+1},
\]
we have
\[
[u^a]\exp(kB(u))
=\sum_{r=0}^a\frac{k^r}{r!}[u^a]B(u)^r.
\]
Since every coefficient of $B(u)$ is nonnegative, every coefficient of this polynomial in $k$ is also nonnegative. Moreover, because $B(u)$ is an odd function, every monomial of $B(u)^r$ has degree congruent to $r$ modulo $2$. Hence, $[u^a]B(u)^r=0$ unless $a\equiv r\pmod 2$, proving property \ref{itm:parity-of-c(a,x)}. For each fixed $a$, the coefficient of $u^a$ in \eqref{eq:d-generating} agrees with this polynomial for infinitely many integers $k$. Thus, it is exactly $c(a,k)$.
\end{proof}

Formula \eqref{eq:general-count} in Theorem \ref{thm:general-formula-for-lattice-points} is a multivariate polynomial in variables $k_1, \dots, k_m$. By Lemma \ref{lem:properties-of-c(a,x)}(\ref{itm:nonnegative-coefficients-c(a,x)}), every coefficient of this multivariate polynomial is nonnegative. We now show, as a further application of Lemma \ref{lem:properties-of-c(a,x)}, that this nonnegativity implies the Ehrhart positivity of Minkowski sums of cross polytopes as stated in Theorem \ref{thm:positivity}.
\begin{proof}[Proof of Theorem \ref{thm:positivity}] Let us write $Q:=Q_G(k_1,\dots, k_m)$ and define $$U:=\bigcup_{i:k_i>0}I_i.$$ Then, by Proposition \ref{prop:dim-of-Q_G}, we have $\dim Q=|U|=d$. By Lemma \ref{lem:properties-of-c(a,x)}(\ref{itm:nonnegative-coefficients-c(a,x)}), every factor $c(a_i,tk_i)$ has nonnegative coefficients in $t$. Therefore, every summand in \eqref{eq:ehrhart-formula}, and hence $i(Q,t)$ itself, has nonnegative coefficients. The constant term is $1$, contributed by the zero draconian sequence.

It remains to show that the coefficient of $t^r$ is positive for every $1\le r\le d$. Choose any subset $W\subset U$ with $|W|=r$. For each $j\in W$, choose a left vertex $i(j)$ such that $j\in I_{i(j)}$ and $k_{i(j)}>0$. We define
\[
a_i:=|\{j\in W\mid i(j)=i\}|.
\]
For every $S\subset[m]$, the quantity $\sum_{i\in S}a_i$ counts selected right vertices assigned to left vertices in $S$, and all such right vertices lie in $\bigcup_{i\in S}I_i$. Thus,
\[
\sum_{i\in S}a_i\le\left|\bigcup_{i\in S}I_i\right|.
\]
This implies $\vct{a}\in D(G)$. Moreover, $|\vct{a}|=r$, and $a_i>0$ only when $k_i>0$.

By Lemma \ref{lem:properties-of-c(a,x)}, the polynomial $c(a_i,tk_i)$ has degree $a_i$ with a positive leading coefficient. Hence, the term
\[
\prod_{i=1}^m c(a_i,tk_i)
\]
is a polynomial of degree $\sum_i a_i=r$ with a positive coefficient of $t^r$. Since every coefficient in every other summand is nonnegative, the coefficient of $t^r$ in the Ehrhart polynomial given in \eqref{eq:ehrhart-formula} is strictly positive. This holds for every $1\le r\le d$.
\end{proof}

\begin{remark}\label{rem:other-formulas}
    The polytopes $Q_G(k_1, \dots, k_m)$ form a subfamily of type B generalized permutohedra (see \cite{Euretall2024}, \cite{Bastidas2021}, and \cite{thawinrak2025AtoB} for comprehensive treatments of type B generalized permutohedra). Consequently, the two equivalent formulas given in \cite[Theorem A]{Euretall2024} and \cite[Corollary 4.11]{thawinrak2025AtoB} for counting the lattice points in type B generalized permutohedra can be applied to $Q_G(k_1, \dots,k_m).$ However, neither formula makes the Ehrhart positivity of $Q_G(k_1, \dots,k_m)$ particularly transparent. Cross polytopes provide an additional illustration of this issue. For example,
\[
\diamondsuit_{[2]}
=
\Delta^0_{[2]}
+\Delta^0_{\{\Bar{1},\Bar{2}\}}
+\Delta^0_{\{\Bar{1},2\}}
-\Delta^0_{\{1\}}
-\Delta^0_{\{\Bar{1}\}}
-\Delta^0_{\{2\}}
-\Delta^0_{\{\Bar{2}\}},
\]
where $\Delta^0_I:=\conv(\vct{0},\vct{e}_i\mid i\in I)$ and $\vct{e}_{\Bar{i}}:=-\vct{e}_i$ for $i\in[n]$, is expressed as a Minkowski sum of simplices with some negative coefficients. Such negative coefficients make the Ehrhart positivity even less apparent when applying the two formulas. In contrast, our formula in \eqref{eq:ehrhart-formula} is better suited to the family $Q_G(k_1,\dots,k_m),$ as it makes the Ehrhart positivity apparent.
\end{remark}

\subsection{Other consequences}

We denote by $D^0(G)$ the set of $G$-draconian sequences $\vct{a}$ whose sum of its components $|\vct{a}|$ is an even integer, and by $D^1(G)$ the set of $G$-draconian sequences whose sum $|\vct{a}|$ of components is an odd integer. Additionally, we define
\[F_{G,0}(z):= \sum_{\vct{a} \in D^0(G)}z^{s(\vct{a})} \quad \text{and} \quad F_{G,1}(z):= \sum_{\vct{a} \in D^1(G)}z^{s(\vct{a})}\]
to be the support-enumerators of $D^0(G)$ and $D^1(G),$ respectively. Clearly, we have $F_G(z) = F_{G,0}(z) + F_{G,1}(z)$. The following two corollaries employ Lemma \ref{lem:properties-of-c(a,x)} to give formulas for the numbers of lattice points in the interior and the boundary of $Q_G(k_1, \dots, k_m)$ in terms of $D^0(G), D^1(G), F_{G,0}(z),$ and $F_{G,1}(z).$ 

\begin{corollary}
Suppose that $Q_G(k_1,\ldots,k_m)$ is $d$-dimensional. Then, the number of lattice points in the (relative) interior of $Q_G(k_1,\ldots,k_m)$ is given by
\begin{equation}
(-1)^d\left(
\sum_{\vct{a}\in D^0(G)}\prod_{i=1}^m c(a_i,k_i)
-
\sum_{\vct{a}\in D^1(G)}\prod_{i=1}^m c(a_i,k_i)
\right).
\label{eq:interior}
\end{equation}
In particular, the number of lattice points in the (relative) interior of $Q_G(1,\dots, 1)$ equals $$(-1)^d(F_{G,0}(2)-F_{G,1}(2)).$$
\end{corollary}

\begin{proof}
By the reciprocity property of Ehrhart polynomials, the number of lattice points in the (relative) interior of $Q:=Q_G(k_1, \dots, k_m)$ equals $(-1)^d i(Q,-1)$. Since $c(a,-k) = (-1)^ac(a,k)$, substituting $t=-1$ into \eqref{eq:ehrhart-formula} and separating the terms according to the parity of $|\vct{a}|$ gives \eqref{eq:interior}.

It is straightforward to check that
\begin{align}\label{eq: c(a,1)}
c(a_i,1)=
    \begin{cases}
       1 &\text{ if } a_i = 0\\
       2 &\text{ if } a_i \geq 1
    \end{cases}.
\end{align}
Thus, plugging \eqref{eq: c(a,1)} into \eqref{eq:interior} gives $(-1)^d(F_{G,0}(2)-F_{G,1}(2))$ for the number of lattice points in the (relative) interior of $Q_G(1,\dots, 1)$.
\end{proof}

Subtracting the lattice points in the interior from all the lattice points in $Q_G(k_1, \dots, k_m)$ gives the following boundary formula.

\begin{corollary}
Suppose that $Q_G(k_1,\ldots,k_m)$ is $d$-dimensional. Then, the number of lattice points on the boundary of $Q_G(k_1,\ldots,k_m)$ is given by
\begin{equation}
\begin{cases}
2\displaystyle\sum_{\vct{a}\in D^1(G)}\prod_{i=1}^m c(a_i,k_i),&d\text{ is even},\\[8pt]
2\displaystyle\sum_{\vct{a}\in D^0(G)}\prod_{i=1}^m c(a_i,k_i),&d\text{ is odd}.
\end{cases}
\label{eq:boundary}
\end{equation}
In particular, the number of lattice points on the boundary of $Q_G(1,\dots, 1)$ is given by
\begin{equation}
\begin{cases}
2F_{G,1}(2),&d\text{ is even},\\
2F_{G,0}(2),&d\text{ is odd}.
\end{cases}
\label{eq:boundary-111}
\end{equation}
\end{corollary}

Lastly, we give the following formula for the surface volume.

\begin{corollary}
Suppose that $Q_G(k_1,\ldots,k_m)$ is $d$-dimensional. Then, the surface volume of $Q_G(k_1,\ldots,k_m)$ is given by
\begin{equation}
\SVol\bigl(Q_G(k_1,\ldots,k_m)\bigr)
=2^d\sum_{\substack{\vct{a}\in D(G)\\|\vct{a}|=d-1}}
\frac{k_1^{a_1}\cdots k_m^{a_m}}{a_1!\cdots a_m!}.
\label{eq:surface-volume}
\end{equation}
\end{corollary}
\begin{proof}
Let $Q:=Q_G(k_1,\ldots,k_m)$. Then,
\[
[t^{d-1}]\,i(Q,t)=\SVol(Q)/2.
\]
Note that every $\vct{a} \in D(G)$ satisfies $|\vct{a}| \leq d$. The Ehrhart polynomial of $Q$ given in \eqref{eq:ehrhart-formula} shows that every $\vct{a}$ satisfying $|\vct{a}|\le d-2$ does not contribute to the coefficient of $t^{d-1}$. Moreover, by Lemma \ref{lem:properties-of-c(a,x)}(\ref{itm:parity-of-c(a,x)}), every $\vct{a}$ such that $|\vct{a}|=d$ also does not contribute. Thus, only $\vct{a}$ such that $|\vct{a}|=d-1$ contributes. Furthermore, Lemma \ref{lem:properties-of-c(a,x)}(\ref{itm:leading-coefficient-c(a,x)}) implies that every such $\vct{a}$ contributes to the coefficient of $t^{d-1}$, and gives the displayed sum in \eqref{eq:surface-volume} after multiplying by $2$.
\end{proof}

\section{Further Questions}\label{sec:further-questions}
Our study of Minkowski sums of cross polytopes reveals several intriguing combinatorial properties of these polytopes. The techniques developed in this paper also suggest several possible directions for future research. Since many aspects of this family of polytopes also remain largely unexplored, we present the following problems to highlight some potential research directions.

\begin{problem}
Let $P_1,\dots, P_m$ be polytopes in $\R^n$. The \emph{Cayley polytope} of $P_1,\dots,P_m$ is 
\[\mathcal{C}(P_1,\dots, P_m):=\conv( (\vct{e}_i,P_i)\mid i \in [m]) \subset \R^{m}\oplus\R^n.\]
When $P_i$ is a simplex $\Delta_{I_i} :=\conv(\vct{e}_j\mid j \in I_i)$, one can show that we recover the definition of a root polytope. Recall that, in Section \ref{sec:lttice-point-enumeration}, we show that a support-enumberator is the $h^*$-polynomial of the corresponding root polytope by manipulating its Ehrhart series. Can this technique be used to give a simple calculation of the $h^*$-polynomials of some other Cayley polytopes?
\end{problem}

\begin{problem}
In \cite{chapoton2026}, Athanasiadis and Chapoton associate to each preorder a corresponding poset, which is equivalent to our definition of $D(G)$, and use this structure to compute the $h^*$-polynomial of the corresponding Minkowski sum of simplices. Can their results be extended to the entire family? Can we generalize the technique for computing the $h^*$-polynomials of Minkowski sums of cross polytopes? Additionally, for which instances do these $h^*$-polynomials admit especially simple or explicit formulas?
\end{problem}

\begin{problem}
    In \cite{chapoton2026}, Athanasiadis and Chapoton define the $h$-polynomial of a preposet, which is equivalent to our notion of a support-enumerator, and formulate several conjectures concerning its properties. In particular, they conjecture that the $h$-polynomial of a preorder is \emph{unimodal} and \emph{palindromic}, and only has real roots. It is easy to see that some of these conjectures, such as palindromicity, do not hold for support-enumerators in general. Nevertheless, this raises the following natural question: For which bipartite graphs $G$ do the support-enumerators of $D(G)$ satisfy the conjectures proposed in \cite{chapoton2026}?
\end{problem}

\begin{problem}
The Minkowski sums of cross polytopes studied in this paper possess rich combinatorial structures beyond their lattice-point enumeration. 
In particular, one may consider their face posets. Note that the face structure of the polytopes in a subfamily of Minkowski sums of simplices is investigated in \cite{PostnikovEtal2008} by Postnikov, Reiner, and Williams through type A generalized permutohedra. 
Given their close connection to Minkowski sums of simplices, can one apply similar tools and techniques in \cite{PostnikovEtal2008} to study the faces of Minkowski sums of cross polytopes? In particular, can one give a formula for their $f$- and \emph{$h$-vectors} (for those simple polytopes)? 
\end{problem}

\section*{Acknowledgement}
We thank Yibo Gao for organizing the PKU Algebraic Combinatorics Experience (PACE 2026) at Beijing International Center for Mathematical Research, Peking University. The research program brought the authors together and made possible the collaboration that led to this work, as well as future joint projects.

We also thank Athanasiadis and Chapoton for their comments, which helped improve the exposition of this paper.



\bibliographystyle{abbrv}
\bibliography{BibContainer}

\begin{description}
        \item[Ziyi Dai\hypertarget{author1}{\textsuperscript{1}}] Department of Mathematics, Peking University, Beijing, China. \\ \texttt{E-mail: 1909506735@qq.com}
    \item[Qilin Hou\hypertarget{author2}{\textsuperscript{2}}] Department of Mathematics, Southern University of Science and Technology, Shenzhen, China.\\ 
    \texttt{E-mail: houql2023@mail.sustech.edu.cn}
    \item[Zhiyuan Liu\hypertarget{author3}{\textsuperscript{3}}] Department of Mathematics, Southern University of Science and Technology, Shenzhen, China. \\ \texttt{E-mail: kracyliu050403@gmail.com}
    \item[Warut Thawinrak\hypertarget{author4}{\textsuperscript{4}}] Beijing International Center for Mathematical Research, Beijing, China. \\ \texttt{E-mail: warutthawinrak@gmail.com}
    \item[Hongyu Wang\hypertarget{author5}{\textsuperscript{5}}] Department of Mathematical Sciences, Tsinghua University, Beijing, China. \\ \texttt{E-mail: hongyuwang2024@outlook.com}
\end{description}

\end{document}